\documentclass[12pt,a4paper,english]{amsart}
\newif\ifdisplaytable
\displaytabletrue 
    
\title{Cuspidal character sheaves on graded Lie algebras II}
\author{Wille Liu}\address{Institute of Mathematics, Academia Sinica, 7F, Astronomy-Mathematics Building, No. 1, Sec. 4, Roosevelt Road, Taipei, Taiwan}\email{wliu@gate.sinica.edu.tw}

\author{Kari Vilonen}\address{School of Mathematics and Statistics, University of Melbourne, VIC 3010, Australia, also Department of Mathematics and Statistics, University of Helsinki, Helsinki, Finland}
\email{kari.vilonen@unimelb.edu.au, kari.vilonen@helsinki.fi}
\thanks{KV was supported in part by the ARC grants  FL200100141, DP250100824 and the Academy of Finland}

\author{Ting Xue}
\address{School of Mathematics and Statistics, University of Melbourne, VIC 3010, Australia, also Department of Mathematics and Statistics, University of Helsinki, Helsinki, Finland} 
\email{ting.xue@unimelb.edu.au}
\thanks{TX was supported in part by the ARC grant  DP250100824}
\date{}

\usepackage{etoolbox}
\usepackage{nicematrix}
\usepackage[inline]{enumitem}

\usepackage{longtable}
\usepackage{aliascnt}
\usepackage{amsmath}
\usepackage{amscd}
\usepackage{amssymb}
\usepackage{amsfonts}
\usepackage{amsthm}
\usepackage{bbm}
\usepackage{cancel}
\usepackage{color}
\usepackage{eucal}
\usepackage{pdfsync}
\usepackage[all,cmtip]{xy}
\usepackage{graphicx}
\usepackage{graphics}
\usepackage[hidelinks]{hyperref}
\usepackage{latexsym}
\usepackage{mathrsfs}
\usepackage{placeins}
\usepackage{bookmark}
\usepackage{mathtools}
\usepackage{stmaryrd}
\usepackage{url}
\usepackage{mathabx}
\usepackage{float}

\usepackage{dynkin-diagrams}
\usepackage{multirow}
\usepackage{cellspace}

\usepackage{xcolor}
\usepackage{tikz-cd}
\usepackage[normalem]{ulem}
\usepackage{manfnt}
\usepackage{svg}

\usepackage{leftindex}
\newcommand{\li}{\leftindex[I]}

  \usepackage{mathrsfs}

  \usepackage{accents}

\usepackage[backend=biber,doi=false,url=false,isbn=false,url=false,clearlang,giveninits=true,style=alphabetic,maxbibnames=99,maxalphanames=6]{biblatex} 
\DeclareRedundantLanguages{english}{english}
\DeclareSourcemap{
  \maps{
    \map{
      \step[fieldsource=language, fieldset=langid, origfieldval, final]
      \step[fieldset=language, null]
    }
  }
}

\DeclareFieldFormat[article]{pages}{#1}
\DeclareFieldFormat[article]{page}{#1}
\DeclareFieldFormat[article]{volume}{\textbf{#1}}
\DeclareFieldFormat[article]{number}{no.~#1}

\DeclareFieldFormat[book]{volume}{\textbf{#1}}
\renewbibmacro{in:}{%
  \ifentrytype{article}{}{\printtext{\bibstring{in}\intitlepunct}}}

\renewbibmacro*{volume+number+eid}{%
  \printfield{volume}%
  \setunit*{\addcomma\space}%
  \printfield{number}%
  \setunit{\addcomma\space}%
  \printfield{eid}}
  
\numberwithin{equation}{section}

    \newcommand{\N}{\mathbb{N}}
    \newcommand{\Z}{\mathbb{Z}}
    \newcommand{\bZ}{\mathbb{Z}}
    \newcommand{\Q}{\mathbb{Q}}
    \newcommand{\R}{\mathbb{R}}
    \newcommand{\C}{\mathbb{C}}
    
    \newcommand{\Cc}{\mathbb{C}^{\times}}

    \newcommand{\cC}{\mathcal{C}}
    \newcommand{\cE}{\mathcal{E}}
    \newcommand{\cF}{\mathcal{F}}

    \newcommand{\cS}{\mathcal{S}}

    \newcommand{\cZ}{\mathcal{Z}}
    \newcommand{\cO}{\mathcal{O}}
    \newcommand{\bA}{\mathbf{A}}
    \newcommand{\bB}{\mathbf{B}}
    \newcommand{\bC}{\mathbf{C}}
    \newcommand{\bD}{\mathbf{D}}

    \newcommand{\bX}{\mathbf{X}}
    \newcommand{\bd}{\mathbf{d}}
    \newcommand{\bs}{\mathbf{s}}

    \mathchardef\mhyphen="2D
    
    \newcommand{\iiA}{\li^2{\mathbf{A}}}
    \newcommand{\iiAI}{\li^2{\mathbf{AI}}}
    \newcommand{\iiAII}{\li^2{\mathbf{AII}}}
    \newcommand{\iiAIII}{\li^2{\mathbf{AIII}}}
    \newcommand{\iiAIIIi}{\li^2{\mathbf{AIII}}\mhyphen\mathrm{i}}
    \newcommand{\iiAIIIii}{\li^2{\mathbf{AIII}}\mhyphen\mathrm{ii}}
    \newcommand{\CI}{\mathbf{CI}}
    \newcommand{\CII}{\mathbf{CII}}
    \newcommand{\CIII}{\mathbf{CIII}}
    \newcommand{\BD}{\mathbf{BD}}
    \newcommand{\BDI}{\mathbf{BDI}}
    \newcommand{\BDII}{\mathbf{BDII}}
    \newcommand{\DII}{\mathbf{DII}}
    \newcommand{\BDIII}{\mathbf{BDIII}}
    \newcommand{\Evi}{\mathrm{E}_6}
    \newcommand{\Evii}{\mathrm{E}_7}
    \newcommand{\Eviii}{\mathrm{E}_8}
    \newcommand{\Fiv}{\mathrm{F}_4}
    \newcommand{\Gii}{\mathrm{G}_2}
    \newcommand{\iiEvi}{\li^2{\mathrm{E}}_6}
    \newcommand{\iiiDiv}{\li^3{\mathrm{D}}_4}

    \newcommand{\im}{\operatorname{im}}

    \usepackage{calligra}
    \DeclareMathOperator{\scrHom}{R\mathscr{H}\text{\kern -3pt {\calligra\large om}}\,}

    \newcommand{\Pf}{\operatorname{\mathcal{P}_f}}

    \newcommand{\nil}{\mathrm{nil}}

    \newcommand{\rc}{\mathrm{c}}
    
    \newcommand{\id}{\operatorname{id}}
    \newcommand{\tr}{\operatorname{tr}}
    
    \newcommand{\supp}{\operatorname{supp}}

    \newcommand{\fc}{\mathfrak{c}}
    \newcommand{\fd}{\mathfrak{d}}
    \newcommand{\fg}{\mathfrak{g}}
    \newcommand{\fh}{\mathfrak{h}}
    \newcommand{\fl}{\mathfrak{l}}
    \newcommand{\ft}{\mathfrak{t}}

    \newcommand{\fp}{\mathfrak{p}}

      \newcommand{\fZ}{\mathfrak{Z}}
    \newcommand{\fgl}{\mathfrak{gl}}
    \newcommand{\fsl}{\mathfrak{sl}}
    \newcommand{\fso}{\mathfrak{so}}
    \newcommand{\fsp}{\mathfrak{sp}}

    \newcommand{\GL}{\operatorname{GL}}
    \newcommand{\SL}{\operatorname{SL}}
    \newcommand{\Sp}{\operatorname{Sp}}
    \newcommand{\SO}{\operatorname{SO}}
    \newcommand{\Spin}{\operatorname{Spin}}

     \newcommand{\Lt}{\mathfrak{t}}
    \newcommand{\Lg}{\mathfrak{g}}
    \newcommand{\La}{\mathfrak{a}}
    \newcommand{\Lh}{\mathfrak{h}}

    \newcommand{\Ad}{\operatorname{Ad}}
    
    \newcommand{\Aut}{\operatorname{Aut}}

    \newcommand{\End}{\operatorname{End}}
    
    \newcommand{\IC}{\operatorname{IC}}
    \newcommand{\Lie}{\operatorname{Lie}}
    \newcommand{\Out}{\operatorname{Out}}

    \newcommand{\Ind}{\operatorname{Ind}}

    \newcommand{\Hom}{\operatorname{Hom}}

    \newcommand{\ord}{\operatorname{ord}}

     \newcommand{\Waff}{W^{\mathrm{aff}}}

     \newcommand{\on}{\operatorname}

    \newcommand{\beqn}{\begin{equation*}}
\newcommand{\eeqn}{\end{equation*}}

\newcommand{\beq}{\begin{equation}}
\newcommand{\eeq}{\end{equation}}

\newcommand{\bega}{\begin{gathered}}
\newcommand{\eega}{\end{gathered}}
 \newcommand{\bern}{\begin{eqnarray*}}
\newcommand{\eern}{\end{eqnarray*}}
\newcommand{\ber}{\begin{eqnarray}}
\newcommand{\eer}{\end{eqnarray}}

\theoremstyle{plain}
\newtheorem{theo}{Theorem}[section]

\newaliascnt{lemm}{theo}
\newaliascnt{prop}{theo}
\newaliascnt{coro}{theo}
\newaliascnt{rema}{theo}
\newaliascnt{defi}{theo}
\newaliascnt{exam}{theo}
\newaliascnt{conj}{theo}
\newtheorem{lemm}[lemm]{Lemma}
\newtheorem{prop}[prop]{Proposition}
\newtheorem{coro}[coro]{Corollary}
\theoremstyle{definition}
\newtheorem{rema}[rema]{Remark}
\newtheorem{exam}[exam]{Example}

\aliascntresetthe{lemm}
\aliascntresetthe{prop}
\aliascntresetthe{coro}
\aliascntresetthe{rema}
\aliascntresetthe{defi}
\aliascntresetthe{exam}
\aliascntresetthe{conj}

\begin{document}
\begin{abstract}
    In this paper we give a complete classification of cyclically  graded semisimple Lie algebras that afford cuspidal character sheaves and determine the support of the cuspidal character sheaves. This constitutes a major step towards the explicit classification of cuspidal character sheaves for graded Lie algebras.
\end{abstract}
\maketitle

\section{Introduction}\label{sec:introduction}
  
This paper is a continuation of our previous work~\cite{LTVX}, where we give a uniform construction of cuspidal character sheaves on cyclically graded Lie algebras.  Making use of results in that paper we proceed here to determine the gradings which afford cuspidal character sheaves and to determine the supports of cuspidal character sheaves explicitly for graded simple Lie algebras. 

Character sheaves on graded Lie algebras are connected to representations of double affine Hecke algebras~\cite{Va05,LY18,Liu23} and cohomology of homogeneous affine Springer fibres as well as cohomology of Hessenberg varieties~\cite{VV09,OY16,CVX17,CVX20b,CVX20a}. They also appear to be closely related to the emerging theory of affine character sheaves~\cite{BC25,NY25}, which aim to provide a geometric tool to study characters of $p$-adic groups. In~\cite{LTVX} we explain the connection of {\em cuspidal} character sheaves on graded Lie algebras with homogeneous elliptic  affine Springer fibres and finite-dimensional representations of double affine Hecke algebras.  One of our goals is to classify cuspidal character sheaves on each graded simple Lie algebra. We expect to apply our classification results  to the study of cohomology of homogeneous affine Springer fibres as proposed in \cite{TX25}.
 \par

  We have shown in~\cite{LTVX} that cuspidal character sheaves always arise from a nearby-cycle construction using  \emph{nil-supercuspidal data} as input, see also~\autoref{ssec:strategy}. This allows us to reduce the problem of classifying  cuspidal character sheaves to the classificaiton of  nil-supercuspidal data and the calculation of endomorphism rings of their associated nearby-cycle sheaves. This calculation has been carried out in most cases previously by the second and third named authors. Combining the results in~\cite{VX23,VX24,X24} with results in this paper we obtain a complete description of cuspidal character sheaves for majority of classical graded Lie algebras.

Let $G$ be a simply connected almost simple algebraic group over $\C$ and $\Lg=\Lie G$. The main results of this paper classify cyclic gradings $\Lg=\oplus_{i\in\bZ/m\bZ}\Lg_i$   which afford cuspidal character sheaves  and determine the supporting strata of cuspidal character sheaves for these gradings. See~\autoref{thm-classical} for classical types and~\autoref{thm-excp} for exceptional types. \par  

To determine the supporting strata of cuspidal character sheaves, recall from~\cite{LTVX} that the existence of cuspidal character sheaves is equivalent to the existence of \linebreak nil-supercuspidal data, which in turn is equivalent to the following statement:
\begin{equation}\label{eqn-levi}
\begin{gathered}
   \text{There exists a $\theta$-stable Levi subgroup $L\subseteq G$ such that $\fl = Z_{\fg}(Z(\fl)_1)$, } 
   \\
 \text{$Z(\fl)_0 = 0$, and $\theta|_{L_{\on{der}}}$ affords a bi-orbital cuspidal character sheaf}  
   \end{gathered}
\end{equation}
 where $Z(\fl)$ denotes the centre of $\fl$ and $Z(\fl)_i=Z(\fl)\cap\Lg_i$. 
More explicitly, let $L$ be a Levi as in~\eqref{eqn-levi}. Let $Z(\fl)_1^\circ=\{s\in Z(\fl)_1\mid Z_\Lg(s)=\fl\}$ and $H=L_{\on{der}}$.  
The supporting strata of cuspidal character sheaves are
\beqn
\cS:=G_0\cdot (Z(\fl)_1^\circ+x_n)\,
\eeqn
where $H_0\cdot x_n\subset\Lh_1^{\on{nil}}$ is the supporting stratum of a {\em bi-orbital cuspidal} character sheaf on $\Lh_1$. Therefore, the description of the strata $\cS$ boils down to determining \begin{enumerate*}[label={(\roman*)}]\item the subspace $Z(\fl)_1$ and \item the supporting stratum $H_0x_n$ of bi-orbital supercuspidal sheaves on $\fh^{\nil}_1$. \end{enumerate*} 
\par

In what follows we call the space $Z(\fl)_1$ as above, and its $G_0$-conjugates, a \emph{cuspidal Cartan subspace} of $\Lg_1$. The $G_0$-saturation of such a cuspidal Cartan subspace  is the semisimple part of the support of a cuspidal character sheaf. 
Note that a cuspidal Cartan subspace is not necessarily a Cartan subspace.  

As a consequence of our classification results, we have the following theorem.

\begin{theo}\label{theo:cuspidal-cartan}
\leavevmode
 \begin{enumerate}
     \item For each $(G,\theta)$, there exists at most one $G_0$-conjugacy class of  
$\theta$-stable Levi subgroups $L$ such that $Z(\fl)_0 = 0$, $\fl = Z_{\fg}(Z(\fl)_1)$ and $\theta|_{L_{\on{der}}}$ affords a bi-orbital cuspidal character sheaf.

 \item All cuspidal Cartan subspaces are $G_0$-conjugate to each other and are of dimension $\dim \fg_1 - \dim \fg_0 $.
        \item         If $\fd\subseteq \fg_1$ is a cuspidal Cartan subspace, then $W_{\fd} := N_{G_0}(\fd) / Z_{G_0}(\fd)$ is a complex reflection group. Moreover, the natural map $\fd /\!/ W_{\fd}\to \fg_1 /\!/ G_0$ is a closed immersion. 
        \item 
        If $\cF$ is a cuspidal character sheaf on $\fg_1$, then $\supp \cF /\!/ G_0 \cong \fd/\!/ W_{\fd}$ for any cuspidal Cartan subspace $\fd\subseteq \fg_1$. 
 \end{enumerate}

\end{theo}

More specifically, let us fix a Cartan subspace $\fc\subset\Lg_1$. We have
\begin{coro}
\leavevmode
\begin{enumerate}
\item If $\dim\Lg_1-\dim\Lg_0=\dim\fc>0$, then all cuspidal character sheaves have full support.
\item
    If $(G, \theta)$ is not of type $\mathbf{BDIII}$, in the Vinberg classification of~\cite{Vi77}, then every cuspidal character sheaf on $\fg_1$ has the same support. 
    \item 
    Assume that $(G, \theta)$ is of type $\mathbf{BDIII}$ and $G \cong \Spin(N)$.  
    Let $\varepsilon\in Z(G)$ be the generator of the kernel of the double cover $\Spin(N)\to \SO(N)$. 
    Every cuspidal character sheaf on $\fg_1$ with trivial (resp. non-trivial) $\varepsilon$-action has the same support, closure of a $G_0$-distinguished nilpotent orbit. 
\end{enumerate}
    
\end{coro}
\begin{rema}
    It follows from our classification that in most cases cuspidal Cartan subspaces are in fact   Cartan subspaces of $\fg_1$, that is, the cuspidal character sheaves have full suport. This  is the case in particular when the action $G_0\curvearrowright\fg_1$ is GIT-stable, i.e., when there exists a semisimple element $s\in \fg_1$ with finite stabiliser $Z_{G_0}(s)$ in $G_0$. In these cases, parts (2)-(4) of \autoref{theo:cuspidal-cartan} are reduced to the classical results of Vinberg~\cite{Vi77}. We refer the readers to~\autoref{coro:weyl} for the list of $W_\fd$ when the cuspidal Cartan subspace $\fd$ is strictly contained in a Cartan subspace.  At the other extreme,  when $\dim\Lg_1=\dim\Lg_0$ the only such Levi subgroup $L$ as in part (1) of the  theorem, if it exists, coincides with $G$.  Thus the only possible cuspidal Cartan subspace is $0$ and therefore the cuspidal character sheaves have nilpotent support. This latter case resembles Lusztig's $m=1$ situation.
\end{rema}
\begin{rema}
Consider the supporting stratum $\cS$ of a cupsidal character sheaf. Let $x = x_s + x_n\in \cS$ with $x_s\in \fd^{\circ}=\{s\in\fd\mid Z_\Lg(s)=Z_\Lg(\fd)\}$. We have  the following exact sequence for the equivariant fundamental group of $\cS$ (see~\cite{LTVX})
\[
    1\to \pi_0(Z_{G_0}(x))\to \pi_1(\cS/G_0, x) \to \pi_1(\fd^{\circ}/W_{\fd}, x_s)\to 1
\]
where $\pi_1(\fd^{\circ}/W_{\fd}, x_s)$ is the braid group, see  Brou\'e--Malle--Rouquier~\cite{BMR95}, associated with  $W_\fd$, which is a complex reflection group by ~\autoref{theo:cuspidal-cartan}. This allows us to describe any $G_0$-equivariant local system on $\cS$, in particular, those giving rise to cuspidal character sheaves, as representations of the braid group of $W_{\fd}$ extended by a finite group $\pi_0(Z_{G_0}(x))$. Such representations associated to cuspidal character sheaves have been determined in most cases of classical types  in~\cite{VX22,VX23,VX24,X24}.

\end{rema}

The paper is organised as follows. In \S\,\ref{sec:overview} we briefly recall the results that we need from~\cite{LTVX} and explain our general classification strategy. A key step is the classification of bi-orbital cuspidal character sheaves, for which we describe the general method in~\S\,\ref{ssec:method}. In \S\,\ref{sec:quiver} we recollect the necessary results on graded Lie algebras. In particular, we give preliminary dimensional estimates in~\S\,\ref{ssec:dim}. In \S\,\ref{sec:statements} we state our main classification results explicitly treating each type separately. For classical types, $\bA, \iiA, \BD, \bC$, we use the quiver description of the graded Lie algebra $\fg_*$ recalled in~\S\,\ref{ssec:gradings} and~\S\,\ref{ssec:quivers}  and the multi-segment notation for nilpotent orbits described in~\S\,\ref{ssec:multiseg}. For exceptional types, we use Kac coordinates and RLYG labelling for graded Lie algebras as recalled in~\S\,\ref{ssec:affine} and labelling of distinguished nilpotent orbits given in~\S\,\ref{ssec:nilp}. The proofs of the main results are given in~\S\,\ref{sec:classical} and~\S\,\ref{sec:exceptional} making use of preparatory results in~\S\,\ref{sec:distinguished} and~\S\,\ref{sec-bicus}. In particular, in~\S\,\ref{sec:distinguished} we determine the $G_0$-distinguished elements in each type and in~\S\,\ref{sec-bicus} we give a classificaiton of bi-orbital cuspidal character sheaves.

\bigskip

\noindent {\bf Acknowledgement.} We thank Cheng-Chiang Tsai for helpful discussions.

\section{Overview}\label{sec:overview}

In this section we  explain the general  strategy that we use to obtain the classification of graded Lie algebras which afford cuspidal character sheaves. The detailed results are stated in~\autoref{sec:statements}. 

\subsection{Graded Lie algebras}
We briefly recall some basic facts about graded Lie algebras and character sheaves as discussed in~\cite{LTVX}. Throughout this paper let $G$ be a simply connected almost simple algebraic group endowed with an automorphism $\theta$ of order $m$ for some $m\ge 1$. The Lie algebra $\fg = \Lie G$ acquires a $\Z/m\Z$-grading $$\fg=\bigoplus_{k\in\bZ/m\bZ}\fg_k$$defined by $\fg_k = \fg^{\theta = \zeta^k_m}$, where $\zeta_m = e^{2\pi \mathbf{i} / m}$. We set $G_0 = G^{\theta}$. As $G$ is simply connected it follows from a theorem of Steinberg that $G_0$ is connected. The adjoint action of $G$ on $\fg$ restricts to an action of $G_0$ on $\fg_k$ for every $k\in \Z/m\Z$. 

We fix a Cartan subspace $\fc\subset\fg_1$. Let  $W(G_0, \fc)=N_{G_0}(\fc)/Z_{G_0}(\fc)$ be the Weyl group of $(G_0,\Lg_1)$ and let 
$$
f:\fg_1 \to \fg_1 /\!/ G_0\cong\fc/W(G_0,\fc)
$$
 denote the adjoint quotient map (\cite{Vi77}). We refer to $\dim\fc$ as the rank of the graded Lie algebra. The nullcone $\Lg_{1}^{\on{nil}}$ in the sense of invariant theory is given by $\Lg_{1}^{\on{nil}}=f^{-1}(f(0))$. We use similar notation for $\Lg_{-1}$. 

A character sheaf is, for the purposes of this paper, a simple $G_0$-equivariant perverse sheaf on $\Lg_1$ with nilpotent singular support. Equivalently, it is the Fourier transform of an irreducible $G_0$-equivariant perverse sheaf on $\Lg_{-1}^{\on{nil}}$, the nilpotent variety of $\Lg_{-1}$.  A character sheaf is called {\em cuspidal} if it does not arise from parabolic induction from smaller groups. A character sheaf is called {\em bi-orbital} if it has nilpotent support. 

In~\cite[\S2]{LTVX} we introduce a stratification of $\fg_1$. The strata are of the form 
$$\cS_{s+n}:=G_0\cdot(\fc_s+n),\,s\in\fc,\ n\in Z_\fg(s)_{\on{der}}\cap\Lg_1^{\on{nil}}$$ 
where $\fc_s = \{x\in \fg_1\mid Z_G(x) = Z_G(s)\}\subseteq\fc$ is a $\fc$-stratum. As character sheaves are Fourier transforms of irreducible $G_0$-equivariant perverse sheaves on $\Lg_{-1}^{\on{nil}}$ they are supported on closures of duals $\widecheck{\cO}$ of nilpotent orbits $\cO\subseteq \Lg_{-1}^{\on{nil}}$. 

In~\cite[\S2]{LTVX} we show that the closure of the support of a character sheaf $\cF$ contains  an open  stratum $\widecheck\cO$ (which is of the form $\cS_{s+n}$) such that  $\cF|_{\widecheck\cO}$ is a local system.

We will determine the gradings on $\Lg$ such that $\Lg_1$ affords a cuspidal character sheaf. For such gradings we determine the supports of the cuspidal character sheaves. We will  rely heavily on the results in~\cite{LTVX}.

\subsection{Overall strategy}\label{ssec:strategy}

We begin by recalling some notions and results from ~\cite{LTVX}. An element $x\in \fg_{1}$ is \emph{$G_0$-distinguished} if the stabilizer $Z_{G_0}(x)$ is unipotent. Or, equivalently, if $x$ is not contained in  a $\theta$-stable Levi factor of a $\theta$-stable proper parabolic subalgebra of $\Lg$. If $x$ is distinguished we also call the $G_0$-orbit $\rc=G_0 \cdot x$ distinguished. 

Let $x_s$ denote the semisimple part of a $G_0$-distinguished element $x\in\fg_1$. Let $H=Z_G(x_s)_{\on{der}}$, $\Lh=\on{Lie}H$ and $\Lh_1=\Lh\cap\Lg_1$. Let $\chi$ be an irreducible $G_0$-equivariant local system on  $\rc=G_0\cdot x$. We call $(\rc,\chi)$ a nil-supercuspidal datum if the intersection cohomology sheaf $\IC(\chi|_{\rc\cap(x_s+\fh_{1}^{\nil})})\in \on{Perv}_{H_{0}}(\fh_{1}^{\nil})$ is a sum of bi-orbital cuspidal sheaves.

In~\cite{LTVX} we have shown that if $\widecheck\cO\subseteq \fg_1$ is the supporting stratum of a cuspidal character sheaf, then 
\begin{enumerate}
\item $\widecheck\cO$ consists of $G_0$-distinguished elements.   
\item $\dim \fg_1 - \dim \fg_0 = \dim f(\widecheck\cO).$

\end{enumerate}
In particular, 
\begin{enumerate}
    \item if $\fg_1$ admits a cuspidal character sheaf, then $\dim \fg_1\ge \dim \fg_0 $.
    \item if $\cF$ is a cuspidal character sheaf on $\fg_1$, then it is bi-orbital if and only if $\dim \fg_1= \dim \fg_0$.
    \item if $\dim\fc=0$, i.e., the rank of the grading is zero, then all cuspidal character sheaves, if they exist, are bi-orbital.
\end{enumerate}

The classification of the gradings affording cuspidal character sheaves is thus reduced to the following cases:
\begin{enumerate}
    \item the case when $\dim\Lg_1=\dim\Lg_0$ and when there exist bi-orbital cuspidal character sheaves;
    \item the case when $\dim\Lg_1>\dim\Lg_0$. In this case the rank has to be positive, i.e., $\dim\fc>0$ in order to have any cuspidals. Then there are the following two possibilities:
\begin{enumerate}
\item  $\dim\Lg_1-\dim\Lg_0=\dim\fc$. In this case  the Levi subgroup $L$ of~\eqref{eqn-levi} is necessarily $G_0$-conjugate to $Z_G(\fc)$. By~\eqref{eqn-levi}  cuspidal character sheaves exist if and only if $\theta|_{Z_G(\fc)_{\on{der}}}$ affords a  bi-orbital cuspidal character sheaf.
\item  $\dim\Lg_1-\dim\Lg_0<\dim\fc$. In this case the Levi subgroup $L$ of~\eqref{eqn-levi} is $G_0$-conjugate to $Z_G(\fc_s)$, where $\fc_s\subset\fc$ is a $\fc$-stratum such that $\dim\fc-\dim\fc_s=\dim\Lg_1-\dim\Lg_0$.  By~\eqref{eqn-levi}  cuspidal character sheaves exist  if and only if $\theta|_{Z_G(\fc_s)_{\on{der}}}$ affords a  bi-orbital cuspidal character sheaf for such a $\fc$-stratum $\fc_s$.
\end{enumerate}
\end{enumerate}
Thus we are reduced to classifying bi-orbital cuspidal character sheaves. We describe the method for that in the next subsection.

\subsection{General method for classifying bi-orbital cuspidal character sheaves}\label{ssec:method}
Let $(\cO, \cC)$ be a pair consisting of a $G_0$-nilpotent orbit $\cO\subseteq \fg_1^{\nil}$ and an irreducible  $G_0$-equivariant local system $\cC$ such that $\IC(\cC)$ is a bi-orbital cuspidal character sheaf. Then the orbit $\cO$ is necessarily $G_0$-distinguished.  Pick an $\fsl_2$-triplet $(e, h, f)$ such that $e\in \cO$, $h\in \fg_0$ and $f\in \fg_{-1}$. Let $\varphi:\Cc\to G_0$ be the cocharacter characterised by $d\varphi(1) = h$. Set 
\[
\fl = \bigoplus_{n\in \Z} \fl_n,\quad \fl_n = \li^{\varphi}_{2n}\fg_n,\quad \fp = \bigoplus_{n\in \Z} \fp_n,\quad \fp_n = \li^{\varphi}_{\ge 2n}\fg_n
\]
where $$\li^\varphi_{n}{\fg}=\{z\in\fg\mid\on{Ad}_{\varphi(t)}z=t^nz,\,\forall\,t\in\Cc\}\text{ and } 
$$
$${}^\varphi_{n}{\fg}_k={}^\varphi_{n}\fg\cap \fg_{\bar k}, \text{ here $\bar{k}\in\Z/m\Z$},\ \ \ \li^{\varphi}_{\ge 2n}\fg_n=(\oplus_{k\geq 2n}{}_k^\varphi\Lg)\cap\Lg_{\bar{n}}$$
and $L = \exp(\fl)$, $L_0 = \exp(\fl_0)$, $P_0 = \exp(\fp_0)$. Then, $\fp_*$ is a spiral of $\fg_*$ and $\fl_*$ is a splitting of $\fp_*$ in the sense of~\cite{LY17a}. By~\cite{Liu24}, the bi-orbital cuspidal condition on $\IC(\cC)$ implies that $(\cO,\cC)$ is a supercuspidal pair, i.e., $\IC(\cC)$ is supercuspidal. In other words, 
\begin{enumerate}
    \item
        $\fl$ is semisimple,
    \item 
        the restriction $(\cO\cap \fl_1, \cC|_{\cO\cap \fl_1})$ is a cuspidal pair on $\fl_1$ (with respect to the $\bZ$-grading on $\fl$) in the sense of~\cite{L95}, and
    \item
            $\IC(\cC) \cong \Ind^{\fg_1}_{\fp_1}\IC(\cC|_{\fl_1\cap \cO})$ (where $\Ind^{\fg_1}_{\fp_1}$ is the spiral induction). 
\end{enumerate}
By~\cite{LTVX}, the bi-orbital cuspidal condition implies that $\dim \fg_1 = \dim \fg_0$. Therefore, our task is to enumerate supercuspidal pairs $(\cO, \cC)$ on graded Lie algebras satisfying $\dim \fg_1 = \dim \fg_0$. \par

By~\cite{L95}, every cuspidal pair on the $\Z$-graded Lie algebra $\fl_*$ is of the form $(\cO_L\cap \fl_1, \cC_L|_{\cO_L\cap \fl_1})$ for some cuspidal pair $(\cO_L, \cC_L)$ on $\fl$ in the sense of~\cite{L84}. The latter is classifed by~\cite{L84,LS85,S85} for simply connected groups. Therefore, the classification of supercuspidal pairs $(\cO, \cC)$ is reduced to checking the action of central character of $L$ on the cuspidal pair $(\cO_L, \cC_L)$. 

We will carry this out in~\S\,\ref{sec-bicus}.

 \section{Recollections on graded Lie algebras and dimension estimates}\label{sec:quiver}

For the purpose of classification, we will often consider pairs $(G, \theta)$ of a fixed \emph{type} simultaneously. Given a simply connected almost simple complex algebraic group $G$, the outer automorphism group $\Out(G) = \Aut(G) / \Ad(G)$ is isomorphic to a permutation group $\mathfrak{S}_a$ for some $a\in \{1, 2, 3\}$; therefore, the conjugacy class of an element of $\Out(G)$ is determined by its order. Given  $G$ of Cartan type $\bX$ and an automorphism $\theta\in \Aut(G)$, we say that $(G,\theta)$ is of type $\li^c\bX$ if the image of $\theta$ under the projection $\Aut(G)\to \Out(G)$ is of order $c$. We will omit $c$ from the notation when $c = 1$ and write simply $\bX = \li^1\bX$. Note that $c\mid \on{ord}(\theta)$. \par

In this section, we recall some basic results on the description of the pair $(G, \theta)$ of a given type. 

\subsection{Gradings of classical Lie algebras}\label{ssec:gradings}
We recall the explicit description of gradings on the Lie algebras for pairs $(G,\theta)$ of classical types $(\bA, \iiA, \bB, \bC, \bD, \li^2\bD)$ following Vinberg~\cite{Vi77}. Let $m=\on{ord}(\theta)$.

The types $\bB$, $\bD$, and $\li^2\bD$ can be treated uniformly; we denote them together as $\BD$. We write $\zeta_k=e^{2\pi\mathbf{i}/k}$ for $k\in\Z_{>0}$. Then, the pairs $(G, \theta)$  of classical types arise as follows.  

\begin{itemize}
\item Type $\bA$.  Let $G=\SL(V)$, where $V$ is a finite-dimensional complex vector space. Let $\gamma\in \GL(V)$ such that $\gamma^m = 1$. Let $\theta(g) = \gamma g \gamma^{-1}$. 
\item Type $\li^2{\bA}$.  Let $G=\SL(V)$, where $V$ is a finite-dimensional complex vector space equipped with a non-degenerate bilinear form $(\relbar,\relbar)$. Let $\theta:G\to G$ be the unique automorphism satisfying
\[
(gv, w) = (v, \theta(g)^{-1}w) \quad \text{for $v, w\in V$\,.}
\]
We have $\theta^2(g)=\gamma g\gamma^{-1}$, where $\gamma\in G$ is characterised by the condition $(v,w) = (w, \gamma v)$ for every $v, w\in V$.  We can assume that $\gamma^{m/2}\in \{ \pm \id_V\}$. 

\item Type $\bC$ (resp. type $\bB\bD$). Let $V$ be a finite-dimensional complex vector space equipped with a non-degenerate symplectic (resp. symmetric bilinear) form $(\relbar,\relbar)$. Let $G=\Sp(V)$ (resp. $\SO(V)$) and define $\theta:G\to G$ by
\[
    \theta(g)=\gamma g\gamma^{-1},\,\gamma\in \Sp(V) \text{ (resp. $\on{O}(V)$)}.
\]
We have that $\gamma^{m}\in \{\pm \id_V\}$. 
\end{itemize}\par

Note that in type $\BD$, the automorphism $g \mapsto \gamma g \gamma^{-1}$ is defined by picking any lift of $\gamma$ in $\on{Pin}(V)$. 

We set $m_0=m/2$ type $\li^2{\bA}$ and $m_0 = m$ for types $\bA$, $\bC$ and $\mathbf{BD}$. Then, we always have $\gamma^{m_0}\in \{\pm \id_V\}$ by assumption. Let $$V_k=\{v\in V\mid \gamma v=\zeta^{2k}_{2m_0} v\}$$ for $k\in \frac{1}{2}\Z/m_0\Z$. Then $V=\bigoplus_{k\in I}V_k$, where 
\begin{equation*}
    \text{$I = \Z/m_0\Z$ (when $\gamma^{m_0} = \id_V$) or $I = (\frac{1}{2}+\Z)/m_0\Z$ (when $\gamma^{m_0} = -\id_V$). }
\end{equation*}
 We have that $(V_j,V_k)=0$ unless $k = -j$ in $I$. Moreover, in type $\li^2\bA$, $(\relbar,\relbar)|_{V_0}$ is a non-degenerate symmetric bilinear form and $(\relbar,\relbar)|_{V_{m_0/2}}$ is a non-degenerate symplectic form.

We have, for $k\in \Z/m\Z$,
\[
\fg_k=\begin{cases}
\{x\in \End(V)\mid \tr(x) = 0,\; xV_{j}\subseteq V_{j+k}\} & \bA \\ 
\{x\in\End(V)\mid \tr(x) = 0,\; (xv,w)+\zeta_{m}^k(v,xw)=0\;\forall v,w\in V\}& \iiA \\
\{x\in \End(V)\mid xV_{j}\subseteq V_{j+k},\; (xv,w)+(v,xw) = 0\;\forall v,w\in V\} & \BD, \bC 
\end{cases}
\]
 \par
Let $\bd = \dim_I V\in \N^I$ denote the dimension vector of $V$. For each of the four families ($\mathbf{A}, \mathbf{BD}, \mathbf{C}, \li^2{\mathbf{A}}$), the $G^{\sigma}$-conjugacy class of $\theta$ is determined by the set $I$ and the dimension vector $\bd$. Therefore, we shall denote
\[
	\fg(\bd)_* = \fg(V)_* = \begin{cases}\fsl(\bd)_* & \mathbf{A}\\ \fso(\bd)_*& \mathbf{BD} \\ \fsp(\bd)_*& \mathbf{C} \\ \li^2{\fsl}(\bd)_* & \li^2{\mathbf{A}}\end{cases}.
\]
Moreover, each of the families $\mathbf{BD},\mathbf{C}, \li^2{\mathbf{A}}$ are subdivided into subfamilies according to conditions on the index set $I$, which Vinberg~\cite{Vi77} calls types I, II and III. We will call them types $\iiAI$, $\CIII$, $\BDII$, etc. The conditions are listed in the following table:
\begin{center}
\begin{tabular}{c|c|c|c}
& $\BD$ & $\bC$ & $\iiA$  \\ \hline
I & $I = \Z/m_0\Z,\; 2\mid m_0$ & $I = (\Z+\frac{1}{2})/m_0\Z,\; 2\mid m_0$ & $I = \Z/m_0\Z,\; 2\nmid m_0$\\ \hline
II & $I = (\Z+\frac{1}{2})/m_0\Z,\; 2\mid m_0$ & $I = \Z/m_0\Z,\; 2\mid m_0$ & $I = (\Z+\frac{1}{2})/m_0\Z,\; 2\nmid m_0$ \\ \hline
III & $2\nmid m_0$ & $2\nmid m_0$ & $2\mid m_0$\\ 
\end{tabular}
\end{center}
For $(G, \theta)$ of type $\bA$, we do not make such a distinction.

\subsection{Quiver description of classical graded Lie algebras} \label{ssec:quivers}
From~\autoref{ssec:gradings}, we see that the pairs $(G,\theta)$ of classical types can be described in terms of representations of a cyclic quiver $Q = (I, (i\to i+1)_{i\in I})$, possibly equipped with a bilinear form, see~\cite{Y} (also~\cite{YY}). 
We give a recollection below: 
\begin{center}
\newenvironment{tabeq}{\renewcommand{\arraystretch}{1.5} \begin{tabular}{l}}{\end{tabular}}
	\begin{longtable}{ccl}
		$\bA$
		&
			\begin{tikzcd}[column sep=1cm, row sep=.25cm]
				  \cdots & V_1 \ar[swap]{l}{x_1} \\ & & V_0 \ar[swap]{lu}{x_0} \\  \cdots \ar{r}{x_{m-2}} & V_{m-1} \ar{ru}{x_{m-1}} 
			\end{tikzcd}
		&
		\begin{tabeq}
			$G_0 = \on{S}(\prod_{i = 0}^{m-1}\GL(V_i))$ \\
			$\fg_1 =\bigoplus_{i = 0}^{m-1}\Hom(V_i,V_{i+1})$\\
   $m\ge 2$
       \end{tabeq}
		\\
		$\iiAI$
		&
			\begin{tikzcd}[column sep=.8cm, row sep=.25cm]
				V_l \ar{dd}{x_l} &\cdots\ar[swap]{l}{x_{l-1}} & V_1 \ar[swap]{l}{x_1} \\ & & & V_0 \ar[swap]{lu}{x_0} \\ V_{-l} \ar{r}{x_{l-1}^*} &\cdots \ar{r}{x_1^*} & V_{-1} \ar{ru}{x_0^*} 
			\end{tikzcd}
		&
		\begin{tabeq}
			$G_0 = \prod_{i = 1}^l\GL(V_i)\times \SO(V_0)$ \\
			$\fg_1 =\bigoplus_{i = 0}^{l-1}\Hom(V_i,V_{i+1})\oplus\on{S}^2V_l^*$\\
   $m=2(2l+1)$
       \end{tabeq}
		\\
		$\CI$
		&
		\begin{tikzcd}[column sep=.7cm, row sep=.8cm]
    V_{l-\frac{1}{2}} \ar{d}{x_{l}} &\cdots\ar[swap]{l}{x_{l-1}} & V_{\frac{3}{2}} \ar[swap]{l}{x_2}& V_{\frac{1}{2}} \ar[swap]{l}{x_1} \\ V_{l+\frac{1}{2}} \ar{r}{x_{l-1}^*}&\cdots \ar{r}{x_2^*}& V_{-\frac{3}{2}} \ar{r}{x_1^*} & V_{-\frac{1}{2}} \ar{u}{x_0}
		\end{tikzcd} 
		&
		\begin{tabeq}
			$G_0=\prod_{i = 0}^{l-1}\GL(V_{i + \frac{1}{2}})$ \\
    $\fg_1=\bigoplus_{i=1}^{l-1}\Hom(V_{i-\frac{1}{2}},V_{i+\frac{1}{2}})$\\$\qquad\quad\oplus\on{S}^2V_{\frac{1}{2}}\oplus\on{S}^2V_{l-\frac{1}{2}}^*$\\
    $m=2l$
		\end{tabeq} 
		\\
		$\BDI$
		&
		\begin{tikzcd}[column sep=.6cm, row sep=.25cm]
			& V_{l-1} \ar[swap]{ld}{x_{l-1}} &\cdots\ar[swap]{l}{x_{l-2}} & V_1 \ar[swap]{l}{x_1} \\ V_l\ar{rd}{x_{l-1}^*} & & & & V_0 \ar[swap]{lu}{x_0} \\ & V_{l+1} \ar{r}{x_{l-2}^*}&\cdots  \ar{r}{x_1^*}& V_{-1} \ar{ru}{x_0^*}
		\end{tikzcd} 
		&
		\begin{tabeq}
			$G_0\xrightarrow{2:1}\prod_{i=1}^{l-1}\GL(V_{i})\times \SO(V_0)\times\SO(V_l)$ \\
			$\fg_1=\bigoplus_{i=0}^{l-1}\Hom(V_{i},V_{i+1})$\\
    $m=2l$
		\end{tabeq}  \\
		$\iiAII$
		&
			\begin{tikzcd}[column sep=.6cm, row sep=.25cm]
				& V_{l-\frac{1}{2}} \ar[swap]{ld}{x_l} &\cdots\ar[swap]{l}{x_{l-1}} & V_{\frac{1}{2}} \ar[swap]{l}{x_1} \\ V_{l+\frac{1}{2}}  \ar{rd}{x_l^*} & & &  \\ & V_{-l+\frac{1}{2}} \ar{r}{x_{l-1}^*} &\cdots \ar{r}{x_1^*} & V_{-\frac{1}{2}} \ar{uu}{x_0} 
			\end{tikzcd}
		&
		\begin{tabeq}
			$G_0=\prod_{i = 0}^{l-1}\GL(V_{i+\frac{1}{2}})\times \Sp(V_{l+\frac{1}{2}})$ \\
			$\fg_1=\bigoplus_{i = 1}^{l}\Hom(V_{i-\frac{1}{2}},V_{i+\frac{1}{2}})\oplus\bigwedge^2V_{\frac{1}{2}}$ \\
    $m=2(2l+1)$
		\end{tabeq}
		\\
		$\CII$
		&
		\begin{tikzcd}[column sep=.6cm, row sep=.25cm]
			& V_{l-1} \ar[swap]{ld}{x_{l-1}} &\cdots\ar[swap]{l}{x_{l-2}} & V_1 \ar[swap]{l}{x_1} \\ V_l\ar{rd}{x_{l-1}^*} & & & & V_0 \ar[swap]{lu}{x_0} \\ & V_{l+1} \ar{r}{x_{l-2}^*}&\cdots  \ar{r}{x_1^*}& V_{-1} \ar{ru}{x_0^*}
		\end{tikzcd} 
		&
		\begin{tabeq}
			$G_0=\prod_{i=1}^{l-1}\GL(V_{i})\times \Sp(V_0)\times\Sp(V_l)$ \\
			$\fg_1=\bigoplus_{i=0}^{l-1}\Hom(V_{i},V_{i+1})$\\
    $m=2l$
		\end{tabeq} 
		\\
		$\DII$
		&
		\begin{tikzcd}[column sep=.6cm, row sep=.8cm]
    V_{l-\frac{1}{2}} \ar{d}{x_l} &\cdots\ar[swap]{l}{x_{l-1}} & V_{\frac{3}{2}} \ar[swap]{l}{x_2}& V_{\frac{1}{2}} \ar[swap]{l}{x_1} \\ V_{l+\frac{1}{2}} \ar{r}{x_{l-1}^*}&\cdots \ar{r}{x_2^*}& V_{-\frac{3}{2}} \ar{r}{x_1^*} & V_{-\frac{1}{2}} \ar{u}{x_0}
		\end{tikzcd} 
		&
		\begin{tabeq}
			$G_0\xrightarrow{2:1}\prod_{i = 0}^{l-1}\GL(V_{i + \frac{1}{2}})$ \\
			$\fg_1=\bigoplus_{i=1}^{l-1}\Hom(V_{i-\frac{1}{2}},V_{i+\frac{1}{2}})$\\$\qquad\quad\oplus\bigwedge^2V_{\frac{1}{2}}\oplus\bigwedge^2V_{l-\frac{1}{2}}^*$\\
    $m=2l$
		\end{tabeq}  \\
        $\iiAIIIi$
		&
		\begin{tikzcd}[column sep=.6cm, row sep=.25cm]
			& V_{l-1} \ar[swap]{ld}{x_{l-1}} &\cdots\ar[swap]{l}{x_{l-2}} & V_1 \ar[swap]{l}{x_1} \\ V_l\ar{rd}{x_{l-1}^*} & & & & V_0 \ar[swap]{lu}{x_0} \\ & V_{l+1} \ar{r}{x_{l-2}^*}&\cdots  \ar{r}{x_1^*}& V_{-1} \ar{ru}{x_0^*}
		\end{tikzcd} 
		&
		\begin{tabeq}
			$G_0=\prod_{i=1}^{l-1}\GL(V_{i})\times \SO(V_0)\times\Sp(V_l)$ \\
			$\fg_1=\bigoplus_{i=0}^{l-1}\Hom(V_{i},V_{i+1})$\\
    $m=4l$
		\end{tabeq} 
		\\
		$\iiAIIIii$
		&
		\begin{tikzcd}[column sep=.6cm, row sep=.8cm]
    V_{l-\frac{1}{2}} \ar{d}{x_l} &\cdots\ar[swap]{l}{x_{l-1}} & V_{\frac{3}{2}} \ar[swap]{l}{x_2}& V_{\frac{1}{2}} \ar[swap]{l}{x_1} \\ V_{l+\frac{1}{2}} \ar{r}{x_{l-1}^*}&\cdots \ar{r}{x_2^*}& V_{-\frac{3}{2}} \ar{r}{x_1^*} & V_{-\frac{1}{2}} \ar{u}{x_0}
		\end{tikzcd} 
		&
		\begin{tabeq}
			$G_0=\prod_{i = 0}^{l-1}\GL(V_{i + \frac{1}{2}})$ \\
			$\fg_1=\bigoplus_{i=1}^{l-1}\Hom(V_{i-\frac{1}{2}},V_{i+\frac{1}{2}})$\\$\qquad\quad\oplus\bigwedge^2V_{\frac{1}{2}}\oplus\on{S}^2V_{l-\frac{1}{2}}^*$\\
    $m=4l$
		\end{tabeq} 
		\\
		$\CIII$
		&
			\begin{tikzcd}[column sep=.8cm, row sep=.25cm]
				V_l \ar{dd}{x_l} &\cdots\ar[swap]{l}{x_{l-1}} & V_1 \ar[swap]{l}{x_1} \\ & & & V_0 \ar[swap]{lu}{x_0} \\ V_{-l} \ar{r}{x_{l-1}^*} &\cdots \ar{r}{x_1^*} & V_{-1} \ar{ru}{x_0^*} 
			\end{tikzcd}
		&
		\begin{tabeq}
			$G_0=\prod_{i = 1}^l\GL(V_i)\times \Sp(V_0)$ \\
     $\fg_1=\bigoplus_{i = 0}^{l-1}\Hom(V_i,V_{i+1})\oplus \on{S}^2V_l^*$ \\
    $m=2l+1$
		\end{tabeq} 
		\\
		$\BDIII$
		&
			\begin{tikzcd}[column sep=.8cm, row sep=.25cm]
				V_l \ar{dd}{x_l} &\cdots\ar[swap]{l}{x_{l-1}} & V_1 \ar[swap]{l}{x_1} \\ & & & V_0 \ar[swap]{lu}{x_0} \\ V_{-l} \ar{r}{x_{l-1}^*} &\cdots \ar{r}{x_1^*} & V_{-1} \ar{ru}{x_0^*} 
			\end{tikzcd}
		&
		\begin{tabeq}
			$G_0\xrightarrow{2:1}\prod_{i = 1}^l\GL(V_i)\times \SO(V_0)$ \\
			$\fg_1=\bigoplus_{i = 0}^{l-1}\Hom(V_i,V_{i+1})\oplus\bigwedge^2V_l^*$ \\
    $m=2l+1$
		\end{tabeq} 
	\end{longtable}
\end{center}

\subsection{Multi-segments and nilpotent orbits in classical types}\label{ssec:multiseg}

In this subsection   
we introduce the multi-segment labelling of nilpotent orbits in classical types.

Fix a positive integer $m_0 \ge 1$ and a coset $\tilde I\subset \frac{1}{2}\Z$ of index $2$ (either $\tilde I = \Z$ or $\tilde I = 1/2 + \Z$). Set $I = \tilde I / \Z\subset \frac{1}{2}\Z / m_0\Z$. A segment on $I$ is a pair of numbers $[a, b]$ with $a,b \in \tilde I$ satisfying $a \le b$, modulo the congruence relation $[a, b] \equiv [a + m_0, b + m_0]$.  A multi-segment on $I$ is a finite formal sum of segments $\bs  = \sum_{[a,b]} c_{[a,b]} [a, b]$ with $c_{[a,b]}\in \N$.  The number $c_{[a,b]}$ is called the multiplicity of $[a,b]$ in $\bs $. We write $[a,b]\in \bs $ if $c_{[a,b]} > 0$. \par

The length of a segment $[a,b]$ on $I$ is the positive integer $b - a + 1$. We call $[a,b]$ odd/even if its length is so. \par

The dual of a segment $[a,b]$ on $I$ is defined to be $[a,b]^* = [-b, -a]$. A segment $[a,b]$ on $I$ is called self-dual if $[a,b]\equiv [a, b]^* \pmod{m_0}$. The dual of a multi-segment $\bs  = \sum_{[a,b]} c_{[a,b]} [a, b]$ is defined to be $\bs ^* = \sum_{[a,b]} c_{[a,b]} [a, b]^*$. 

Given a multi-segment $\bs  = \sum_{[a,b]} c_{[a,b]} [a, b]$ on $I$, we let
$$\mathbf{d}_{\bs }=(d_i)_{i\in I},\,d_i:=\sum_{[a,b]\in\bs }\sum_{\substack{j\in[a,b]\\j\equiv i\!\!\mod |I|}}c_{[a,b]}$$
where we regard $[a,b]$ as the set of integers $\{a,a+1,a+2,\ldots,b\}$.

We consider the cyclic quiver $Q = (I, (i \to i + 1)_{i\in I})$ as in~\autoref{ssec:quivers}. We label the nilpotent $G_0$-orbits in $\Lg(\mathbf{d})_1$ by multi-segments $\bs $ on $I$ such that $\mathbf{d}_{\bs }=\mathbf{d}$. Given a multi-segment $\bs  = \sum_{[a,b]} c_{[a,b]} [a, b]$, we write $\cO_{\bs }$ for the locally closed subset of $\fg_1$ formed by the nilpotent elements labelled by $\bs$.  Let $x\in\cO_{\bs}$. The Jordan basis of $x$ in $V$ corresponding to $[a,b]\in\bs $ can be chosen as $\{x^iv_a\mid 0\leq i\leq b-a\}$, where $v_a\in V_a$.  In particular, the sizes of Jordan blocks of $x\in\cO_{\bs }$ are given by
$$\sum(b-a+1)^{c_{[a,b]}}$$
where the superscript denotes the multiplicity of $b-a+1$.

In types $\iiA$, $\bC$ and $\BD$, the multi-segments that label nilpotent elements are characterised as follows:
\begin{lemm}\label{lem:multi-seg}
    A multi-segment $\bs  = \sum_{[a,b]} c_{[a,b]} [a, b]$ with $\mathbf{d}_{\bs }=\mathbf{d}$ labels elements in $ \fg(\mathbf{d})_1$ if and only if $\bs  = \bs ^*$ and $2\mid c_{[a,b]}$ whenever
    \[
    \begin{cases}
    2\not\mid (b - a) & \BD; \\
    2\mid (b - a) & \bC; \\
    [2\not\mid (b - a) ]\wedge [m\mid(a+b)] \text{ or } [2\mid (b - a) ]\wedge [m\not\mid(a+b)] & \iiA. \\
    \end{cases}
    \]
    Moreover, $\cO_{\bs}$ forms a single $G_0$-orbit unless $c_{[a,b]} \neq 0 \Rightarrow 2 \nmid (b - a)$ in type $\BD$. 
\end{lemm}
\begin{proof}
We define a contragredient action of $\fg$ on the linear dual $V^*$ by $(xf)(v) = -f(xv)$ for $x\in \fg$ and $(v, f)\in V\times V^*$. 
Let $x\in \fg_1^{\nil}$. We define an operator $x^*$ on $V^*$ by 
\[
    (x^*f)(v) = \begin{cases}-f(xv) & \BD, \bC \\ -\zeta_mf(xv) & \iiA\end{cases},\quad  \text{for $(v, f)\in V\times V^*$}.  
\]
If $V = \bigoplus_{k} M_k$ is a decomposition into indecomposable $I$-graded $\C[x]$-modules, then $V^* = \bigoplus_{k} M^*_k$ with $M^*_k = \bigcap_{l\neq k} (M_l)^{\perp}$ is a decomposition into indecomposable $I$-graded $\C[x^*]$-modules. Let $\bs $ be the multi-segment associated with $(V, x)$, then the dual $\bs ^*$ is the multi-segment associated with $(V^*, x^*)$. 
The bilinear form $(, )$ on $V$ being non-degenerate, induces an isomorphism $V\xrightarrow{\sim} V^*,\; v\mapsto (v,\relbar)$ which intertwines $x$ and $x^*$ and induces $V_i \cong (V_{-i})^*$ for $i\in I$. It follows that $\bs  = \bs ^*$. \par
The well-known parity conditions of partitions for symplectic and orthogonal groups generalise easily to that of the segments in the graded setup. We leave the details to the reader.  
\end{proof}

\subsection{Affine root systems and Kac coordinates}\label{ssec:affine}
For $(G, \theta)$ of exceptional type, it is more convenient to describe $\fg_*$ in terms of affine root systems and Kac coordinates. See~\cite[\S 2]{RLYG} and~\cite[\S 3]{LY18} for relations between affine root systems and graded Lie algebras. \par
Let   $E= (B, T, a:U/[U,U]\to \mathbb{G}_a)$ a pinning for $G$. Let $\sigma\in \Aut_E(G)$ be a pinned automorphism. We have $\ord \sigma \in \{1, 2, 3\}$. Set $T_0 = T^{\sigma}$. Then, we can associate with $(G, \sigma, T)$ an affine root system $(V, \Phi)$ and an affine Weyl group $\Waff$, where $V = \mathbf{X}_*(T_0)\otimes_{\Z} \Q$ and $\Phi$ is a set of affine functions on $V$.  

The choice of the Borel subgroup $B$ yields an affine basis $\Delta\subset \Phi$. There is a unique family of positive integers $(b_{\alpha})_{\alpha\in \Delta}$ determined by the conditions:
\beq\label{eqn-bi}
\sum_{\alpha\in \Delta}b_{\alpha}\alpha = \frac{1_V}{\ord \sigma}\quad\text{and}\quad \gcd\{b_{\alpha} \mid \alpha\in \Delta\} = 1.
\eeq
Here, $1_V$ is the constant function with value $1$ on $V$. \par

Each point $x\in V$ determines a system of coordinates $\left(\alpha(x)\right)_{\alpha\in\Delta}$. Let $m_x\ge 1$ be the least common denominator of the set $\{\alpha(x)\}_{\alpha\in\Delta}$. We obtain a $\Z/m_x\Z$-grading on $\fg$, denoted by $\fg_{x,*}$, which satisfies $\fg_{\alpha}\subseteq \fg_{x, m_x\alpha(x)}$. We set $G_{x,0} = e^{\fg_{x,0}}\subseteq G$. The system $\left(m_x\alpha(x)\right)_{\alpha\in\Delta}$ is called the Kac coordinates of $x$. \par
The fundamental alcove associated with $\Delta$ is defined by $$\kappa = \{x\in V\mid\alpha(x) \ge 0\;,\; \alpha\in \Delta\}.$$ 

\subsection{Labelling of distinguished nilpotent orbits}\label{ssec:nilp}
 
Let $(G, B, T, \sigma, V,\Phi, \Delta)$ be as in~\autoref{ssec:affine}. Let $x\in V$ and write $m = m_x$ and $\fg_* = \fg_{x, *}$. Given a $G_0$-distinguished nilpotent orbit $\cO\subseteq \fg_1^{\nil}$, we pick an $\fsl_2$-triple $(e, h, f)$ such that $e\in \cO$, $h\in \fg_0$ and $f\in \fg_{-1}$. Let $\varphi:\Cc\to G_0$ be the cocharacter characterised by $d\varphi(1) = h$. We may choose $(e, h, f)$ such that $\im\varphi\subseteq T_0$.

Set 
\[
\fl = \bigoplus_{n\in \Z} \fl_n,\quad \fl_n = \li^{\varphi}_{2n}\fg_n,\quad L = \exp(\fl),\quad L_0 = \exp(\fl_0).  
\]
where $$\li^\varphi_{n}{\fg}=\{z\in\fg\mid\on{Ad}_{\varphi(t)}z=t^nz,\,\forall\,t\in\Cc\}\text{ and } 
{}^\varphi_{n}{\fg}_k={}^\varphi_{n}\fg\cap \fg_k, \text{ for $k\in\Z/m\Z$}.$$
The $\bZ$-grading on $\fl$ is $1$-rigid in the sense of~\cite{L95}, and the orbit $\cO_L = \Ad_L e$ satisfies $\cO_L\cap \fl_1 = \cO \cap \fl_1 = $ the unique open $L_0$-orbit in $\fl_1$. 
 
The intersection $B_L := B\cap L$ is a Borel subgroup of $L$ containing the maximal torus $T_0$. It gives rise to a basis $\Delta_L$ of the root system $\Phi(L, T)$. We may identify $\Delta_L$ as a subset $\Delta_L\subset\Delta$ with $\#(\Delta \setminus \Delta_L) = 1$. In addition, we may choose $(e, h, f)$ such that $\langle \alpha, \varphi\rangle\ge 0$ for every $\alpha\in \Delta_L$. Then, we have $\langle \alpha, \varphi\rangle\in \{0, 2\}$ for $\alpha\in \Delta_L$. We define the weight function associated with $\cO_L$ by
\beq\label{eqn-rho}
\rho:\Delta_L\to \{0,1\},\quad \rho(\alpha) = \langle \alpha, \varphi\rangle/2. 
\eeq
It is known (see~\cite[\S 5.6]{C93}) that $\rho$ is independent of the choice of the $\fsl_2$-triple $(e, h, f)$ and the Borel pair $(B, T)$, and it determines the orbit $\cO_L$, as well as the grading $\fl_*$.  \par
We summarise the above discussion as follows:
\begin{lemm}\label{lem:dist-orbit}
    There is a well-defined injective map
    \[
        \{\text{$G_0$-distinguished orbit }\cO\in \fg^{\nil}_1/G_0 \}\to \bigsqcup_{\Delta'\subset \Delta} \{0,1\}^{\Delta'},\quad \cO \mapsto (\Delta_L, \rho).
    \]
\end{lemm}

\subsection{Dimension estimates}\label{ssec:dim}

In this subsection we give some dimension estimates for later use and deduce some preliminary consequences in classical types.

Recall that a grading is called GIT stable if there exists a semisimple element $s\in\Lg_1$ such that $Z_{G_0}(s)$ is finite. Such gradings have been classified in~\cite{RLYG}. In what follows we use either Kac diagram or the label in~\cite{RLYG}, which we refer to as RLYG label, to denote the gradings on exceptional Lie algebras.

\subsubsection{Classical types}
\begin{lemm}\label{lem:dim}
\begin{enumerate}
\item In type $\mathbf{A}$ with $G=\SL_N$ (and $m\geq 2$), there exists a non-orbital cuspidal character sheaf  only if $m|N$ and $\bd=(N/m,\ldots,N/m)$.
\item There are no cuspidal character sheaves in type $\li^2{\mathbf{A}}\mathbf{BCDII}$ since $\dim\Lg_1<\dim\Lg_0$.
\item In type $\li^2{\mathbf{A}}\mathbf{BCDIII}$, all cuspidal character sheaves are bi-orbital, if exist, since $\dim\Lg_1\leq\dim\Lg_0$.
\end{enumerate}
\end{lemm}
\begin{proof}
Recall $d_i = \dim V_i$, $i\in I$. 
\begin{enumerate}
\item{Type $\mathbf{A}$.} 
$\displaystyle{\dim\Lg_1-\dim\Lg_0=\sum_{i=0}^{m-1}(d_id_{i+1}-d_i^2)+1=1-\frac{1}{2}\sum_{i=0}^{m-1}(d_i-d_{i+1})^2.}$

Thus $\dim\Lg_1>\dim\Lg_0$ if and only if $d_0=\ldots=d_{m-1}$. 

\item Type $\iiAII$.  
$\displaystyle{\dim\fg_0-\dim\fg_1=\frac{1}{2}\sum_{i=1}^{l}(d_{i-1/2}-d_{i+1/2})^2+\frac{d_{1/2}+d_{l+1/2}}{2}>0}.
$

\item {Type $\CII$.} 
$\displaystyle{\dim\fg_0-\dim\fg_1=\frac{1}{2}\sum_{i=1}^{l-1}(d_i-d_{i-1})^2+\frac{d_l+d_0}{2}>0}.$

\item{Type $\BDII$.} 
$\displaystyle{\dim\fg_0-\dim\fg_1=\frac{1}{2}\sum_{i=1}^{l-1}(d_{i+1/2}-d_{i-1/2})^2+\frac{d_{1/2}+d_{l-1/2}}{2}>0.}$

\item{Type $\iiAIIIi$.} $\displaystyle{\dim\fg_0-\dim\fg_1 
=\frac{1}{2}\sum_{i=0}^{l-1}\left((d_i-d_{i+1})^2+(d_i-d_{i+1})\right)\geq 0}$.

\item{ Type $\iiAIIIii$.}  
$\displaystyle{\dim\fg_0-\dim\fg_1=\frac{1}{2}\sum_{i=1}^{l-1}\left((d_{i+1/2}-d_{i-1/2})^2-(d_{i+1/2}-d_{i-1/2})\right)\geq 0}$.

\item{ Type $\mathbf{BCDIII}$.}
$\displaystyle{\dim\fg_0-\dim\fg_1=\frac{1}{2}\sum_{i=1}^{l}\left((d_i-d_{i-1})^2+(d_i-d_{i-1})\right)\geq 0\,.}$

\end{enumerate}

 \end{proof}

\subsubsection{Exceptional types}
We have the following non GIT stable gradings with positive rank that satisfy $r:=\dim\fg_1-\dim\fg_0\geq 0$. 
\begin{longtable}{p{1cm}|p{7cm}|p{4cm}|p{2.5cm}}
\hline
Type&$r=0$&$r=\dim\fc$&$0<r<\dim\fc$\\
\hline
$F_4$&$\xymatrix@C=6pt{{\substack{0\\\circ\\\,}}\ar@{-}[r]&{\substack{1\\\bullet\\\,}}\ar@{-}[r]&{\substack{0\\\bullet\\\,}}\ar@2{=>}[r]&{\substack{0\\\bullet\\\,}}\ar@{-}[r]&{\substack{1\\\bullet\\\,}}}\ (m=4)$&&\\
\hline
$E_6$&$4_a,4_b,5_a,8_a,8_b$&&$2_a$\\
\hline
$^2E_6$&$8_f,8_c,10_a,10_c$&$10_b$&\\
\hline
$E_7$&$5_a,7_a,8_b,9_a,9_b,12_b$&$4_a,8_a,10_a,10_b,12_a$&$3_a$\\
\hline
$E_8$&$ 4_b,6_b,7_a,8_b,8_c,8_f,9_c,10_c,10_d$, $14_b,14_c,14_d,18_b,18_d$&$9_a,12_e,14_a,18_c$\\
\hline
\caption{Non GIT stable positive rank gradings with $r\geq 0$}\label{table 2}
\end{longtable}
We can compute the dimensions $\dim\Lg_1$ and $\dim\Lg_0$ using~\cite[Proposition 8.6]{Ka}, which says the following
\begin{enumerate}
\item Let $i_1,\ldots,i_a$ be the indices of the affine Dynkin diagram such that $s_{i_j}=0$. Then the Lie algebra $\Lg_0\cong\cZ_{\Lg_0}\oplus(\Lg_0)_{\on{der}}$, where $\cZ_{\Lg_0}$ is the center, which is of $\ell-a$ dimensional, and $(\Lg_0)_{\on{der}}$ is a semisimple Lie algebra whose Dynkin diagram is the subdiagram of the affine Dynkin diagram consisting of the vertices $i_1,\ldots,i_a$. 
\item Let $j_1,\ldots,j_b$ be the indices of the affine Dynkin diagram such that $s_{j_k}=1$. Then the $\Lg_0$-module $\Lg_1$ is isomorphic to a direct sum of $b$ irreducible modules with highest weights $-\alpha_{j_1},\ldots,-\alpha_{j_b}$. 
\end{enumerate}
\begin{exam}The grading $10_b$ of ${}^2E_6$. The Kac diagram is $\xymatrix{1\,1\,0&1\,0\ar@{=>}[l]}$.
So we have $\Lg_0\cong\cZ_{\Lg_0}\oplus(\Lg_0)_{\on{der}}$ where $\on{dim}\cZ_{\Lg_0}=2$ and $(\Lg_0)_{\on{der}}\cong\mathfrak{sl}_2\oplus\mathfrak{sl}_2$. Thus $\dim\Lg_0=8$.

As an $\mathfrak{sl}_2\oplus\mathfrak{sl}_2$-module, $\Lg_1\cong\mathbf{1}\oplus L(\omega_1)\boxtimes\mathbf{1}\oplus L(2\omega_1)\boxtimes L(\omega_1)$, where $L(\lambda)$ denotes the irreducible module of $\mathfrak{sl}_2$ with highest weight $\lambda$, and $\omega_1$ denotes the fundamental weight. Thus $\dim\Lg_1=9$.
\end{exam}

\section{Main results}

\label{sec:statements}
In this section we state our classification results explicitly type-by-type. We classify the gradings that afford cuspidal character sheaves and determine the support(s) of cuspidal character sheaves. In particular, we deduce~\cite[Conjecture 4.8]{X24} as a consequence (see~\autoref{cons-type A}), which determines the cuspidal character sheaves in type $\bA$.

When $G=\Spin(V)$, we write $\varepsilon = \varepsilon_V\in Z(G)$ for the generator of the kernel of the double cover map $\Spin_N\to \SO_N$.

We show that all {\em non-orbital} cuspidal character sheaves have full support (i.e., support equals $\Lg_1$) unless $G$ is of type $A$, $B$, $D$, $E_6$, or $E_7$. In the latter case the cuspidal character sheaves have non-trivial central characters, i.e., $Z(G)^\theta$ acts non-trivially. In particular, when $G=\Spin_N$, if a {\em non-orbital} cuspidal character sheaf  does not have full support, then  it has non-trivial $\varepsilon\in Z(G)$ action. \par

Throughout this section, we let $(G, \sigma, \theta, m, \fg_*, \fc)$ be as in~\autoref{sec:quiver}. 
We describe the gradings that afford cuspidal character sheaves in the following theorems. In each case, we also determine the support of the cuspidal character sheaves.

\subsection{Gradings affording cuspidal character sheaves and supporting strata of cuspidal character sheaves}
Recall $$r=\dim\fg_1-\dim\fg_0$$ and the notations for cyclic quivers $\fg_* = \fg(\bd)_*$ and $\bd = \dim_I V$ from~\autoref{ssec:gradings} and multi-segments $\bs$ from~\autoref{ssec:multiseg}.
\begin{theo}\label{thm-classical}
Let $(G, \theta)$ be of classical type.  
The gradings $\fg_*$ 
that afford a cuspidal character sheaf and the support of the cuspidal character sheaves are as follows:
\begin{enumerate}
\item {\bf Type $\bA$} Let $(G, \sigma)=(\SL_N,\id_G)$ and $m\geq 2$. 
\begin{enumerate}[label=(\alph*),font=\itshape]
\item $m\nmid N$ and $\bd=\bd_{\bs }$, where $\bs =[a,N+a-1]$ for some $a$. In this case,
\begin{enumerate}
\item[{\rm(i)}] $r=0$
\item[\rm{(ii)}] the support of the cupsidal character sheaves is $\bar \cO_{\bs }$.
\end{enumerate}
\item $m\mid N$ and the dimension vector  $\mathbf{d}=(d,d,\cdots,d)$ with $d=N/m$. In this case,
\begin{enumerate}
\item[{\rm(i)}] $r=1$
\item[{\rm(ii)}] the supporting stratum of the cuspidal character sheaves is 
$$\cS=G_0\cdot (\fc_s+n),\ n\in\cO_{\mathbf{s}},\  \bs =\sum_{i=0}^{m-1}[i,d+i-1]$$
where $\fc_s\subset\fc$ is the unique $\fc$-stratum such that $Z_G(\fc_s)\cong \operatorname{S}(\GL_d^{\times m})$  and such that $ \theta|_{Z_G(\fc_s)_{\on{der}}}$ permutes cyclically the $m$-factors of $\SL_d$.
\end{enumerate}
\end{enumerate}
\item {\bf Type $\iiAI$ and $\CI$}
The gradings described in~\cite{VX24}, i.e., the gradings such that $\mathbf{d}=\mathbf{d}_{\bs }$, where 
$$\bs =\begin{cases}\displaystyle{\sum_{k=0}^{2l}r[k,k]+\sum_{i=0}^{p-1}[l-i,l+1+i]+\sum_{j=0}^{q-1}[-j,j]}&\text{$\iiAI$}\\
\displaystyle{\sum_{k=0}^{2l-1}r[k+\tfrac{1}{2},k+\tfrac{1}{2}]+\sum_{i=0}^{p-1}[l-i-\tfrac{1}{2},l+i+\tfrac{1}{2}]+\sum_{j=0}^{q-1}[-j-\tfrac{1}{2},j+\tfrac{1}{2}]}&\text{$\CI$}\end{cases}$$
for some $p,q\in\mathbb{Z}_{\geq0},$ and $p+q\leq l$ in type $\iiAI$, $p+q\leq l-1$ in type $\CI$.  In these cases 
\begin{enumerate}
\item[{\rm(i)}]
$r=\dim\fc$
\item[\rm{(ii)}] all cuspidal character sheaves have full support.

\end{enumerate}
\item {\bf Type $\BDI$} Let $G=\Spin_N$.
\begin{enumerate}[label=(\alph*),font=\itshape]
\item The gradings described in~\cite[\S8.3]{VX24}. That is, $\mathbf{d}=\mathbf{d}_{\bs }$, where 
$$\bs =\displaystyle{\sum_{k=0}^{2l-1}r[k,k]+\sum_{i=0}^{p-1}[l-i,l+i]+\sum_{j=0}^{q-1}[-j,j]}$$
for some $p,q\in\mathbb{Z}_{\geq0},$ and $p+q\leq l$. In this case 
\begin{enumerate}
\item[{\rm(i)}]
$r=\dim\fc$
\item[{\rm(ii)}] all cuspidal character sheaves have full support and have trivial action by $\varepsilon\in Z(G)$.
\end{enumerate}
\item  The gradings such that $\mathbf{d}=\mathbf{d}_{\bs }$, where \[\quad\bs =\sum_{k=0}^{2l-1}r[k,k]+\sum_{i=0}^{\lfloor\frac{a-1}{2}\rfloor}[-2i-\delta_a,2i+\delta_a]+\sum_{i=0}^{\lfloor\frac{b-1}{2}\rfloor}[l-2i-\delta_b,l+2i+\delta_b]\]  for some $a,b\in\bZ_{\geq 0}$ such that
$$\text{either $a+b\leq l$, or ($2\mid(l-a-b)$ and $|a-b|\leq l$}).$$
 Here $\delta_n:=n-1-2\lfloor\frac{n-1}{2}\rfloor$. In this case
 \begin{enumerate}
\item[{\rm(i)}]$r=\dim\fc$ if and only if $a+b\leq l$.
  
 \item[{\rm(ii)}] All cuspidal character sheaves have non-trivial $\varepsilon$-action. They are supported on $\overline\cS$, where
$$\cS=G_0\cdot(\fc_s+n),\ n\in\cO_{\bs}$$
and $\fc_s\subset\fc$ is the unique stratum such that $Z_G(\fc_s)$ is totally ramified,  $Z_G(\fc_s)_{\on{der}}\cong\Spin_{N-2lr}$  and such that $ \theta|_{\fh:=Z_\fg(\fc_s)_{\on{der}}}$ has $\mathbf{d}_{\fh}=\mathbf{d}_{\bs-\sum_{k=0}^{2l-1}r[k,k] }. $
 
\end{enumerate}
\end{enumerate}

\item{\bf Type $\iiAIII$ and $\CIII$} The gradings such that $\mathbf{d}=\mathbf{d}_{\bs }$, where
$$\bs =\begin{cases}\displaystyle{\sum_{i=0}^k[-i,i]}&\text{type $\iiAIIIi$}\\\displaystyle{\sum_{i=0}^k[l-\tfrac{1}{2}-i,l+\tfrac{1}{2}+i]}&\text{type $\iiAIIIii$}\\\displaystyle{\sum_{i=0}^k[l-i,l+1+i]}&\text{type $\CIII$}\,\end{cases}$$
for some $k\in\mathbb{Z}_{\geq 0}$.
 In this case 
 \begin{enumerate}
\item[{\rm(i)}]
$r=0$
\item[{\rm(ii)}] there is a unique cuspidal character sheaf supported on the closure of $\cO_{\bs }$. 
\end{enumerate}

\item{\bf Type $\mathbf{BDIII}$} Let $$\bs_0=\sum_{i=0}^k[-i,i],\,\bs_1=\sum_{i=0}^{\lfloor\frac{k-1}{2}\rfloor}[-2i-\delta_k,2i+\delta_k]$$
for some $k\in\mathbb{Z}_{\geq 0}$.

\begin{enumerate}[label=(\alph*),font=\itshape]
\item The gradings such that $\mathbf{d}=\mathbf{d}_{\mathbf{s_0}}$. In this case \begin{enumerate}
\item[{\rm(i)}]
$r=0$
\item[{\rm(ii)}] there is a unique cuspidal character sheaf (with trivial $\varepsilon$-action), supported on the closure of $\cO_{\bs_0}$.
\end{enumerate}
\item The gradings such that $\mathbf{d}=\mathbf{d}_{\mathbf{s_1}}$. In this case
\begin{enumerate}
\item[{\rm(i)}]
$r=0$
\item[{\rm(ii)}]the cuspidal character sheaves are supported on the closure of $\cO_{\bs_1}$  and they have non-trivial $\varepsilon$-action.
\end{enumerate}

\end{enumerate}
\end{enumerate}

\end{theo}
The proof of~\autoref{thm-classical} is given in~\S\ref{sec:classical}.

\begin{coro}\label{cons-type A}
    Conjecture 4.8 in~\cite{X24} holds. That is, the cuspidal character sheaves in type $\bA$ are as described in {\em loc.cit.}
\end{coro}
\begin{proof}
    Part (1) of~\autoref{thm-classical} implies that the nil-supercuspidal data are exactly the input of the nearby cycle construction in~\cite[Proposition 4.7]{X24} in the case of $k=1$.
\end{proof}

Before we proceed to the exceptional types, we introduce the following notation. Let $\cO\subset\fg_{-1}^{\on{nil}}$ be a nilpotent $G_0$-orbit. Pick a normal $\fsl_2$-triple $\phi=(e,f,h)$ such that $e\in\cO$ and $h\in\Lg_0$. We write $G^\phi=Z_G(\phi)=Z_G(e)\cap Z_G(h)$ and $(G^\phi)^0$ the identity component of $G^\phi$.

\begin{theo}\label{thm-excp}
Let $(G,\theta)$ be of exceptional type. The gradings affording cuspidal character sheaves are
\begin{enumerate}
\item all GIT stable gradings (where $r=\dim\fc =$ rank of $\fg_*$): 
\begin{longtable}{p{1cm}|p{13cm}}
\hline
Type&Kac diagram/RLYG label\\
\hline
$G_2$&$\xymatrix@C=1em{{{0}}&{{1}}\ar@{-}[r]\ar@3{->}[l]&{{0}}}\,(m=2)$, $\xymatrix@C=1em{{{0}}&{{1}}\ar@{-}[r]\ar@3{->}[l]&{{1}}}\,(m=3)$, $\xymatrix@C=1em{{{1}}&{{1}}\ar@{-}[r]\ar@3{->}[l]&{{1}}}\,(m=6)$\\
\hline
$^3D_4$&$\xymatrix@C=1em{{{0}}&{0}\ar@{-}[l]&{1}\ar@3{->}[l]}\,(m=3)$, $\xymatrix@C=1em{{{1}}&{0}\ar@{-}[l]&{1}\ar@3{->}[l]}\,(m=6)$, $\xymatrix@C=1em{{{1}}&{1}\ar@{-}[l]&{1}\ar@3{->}[l]}\,(m=12)$\\
\hline
$F_4$&$\xymatrix@C=6pt{{\substack{0\\\circ\\\,}}\ar@{-}[r]&{\substack{1\\\bullet\\\,}}\ar@{-}[r]&{\substack{0\\\bullet\\\,}}\ar@2{=>}[r]&{\substack{0\\\bullet\\\,}}\ar@{-}[r]&{\substack{0\\\bullet\\\,}}}\ (m=2)$, $00100,\ (m=3)$, $10100\,(m=4)$, $10101\,(m=6)$, $11101\,(m=8)$, $11111\,(m=12)$\\
\hline\\
$E_6$&$3_a,6_a,9_a,12_a$\\
\hline
$^2E_6$&$2_a,4_b,6_a,12_b,18_a$\\
\hline
$E_7$&$2_a,6_a,14_a,18_a$\\
\hline
$E_8$&$2_a,3_a,4_a,5_a,6_a,8_a,10_a,12_a,15_a,20_a,24_a,30_a$
\\
\hline
\caption{GIT stable gradings}
\end{longtable}

\item the following non GIT stable gradings with $r > 0$

\begin{longtable}{p{1cm}|p{3.5cm}|p{3.5cm}}
\hline
Type&$r=\dim\fc$&$0<r<\dim\fc$\\
\hline
$E_6$&&$2_a$\\
\hline
$^2E_6$&$10_b$&\\
\hline
$E_7$&$4_a,8_a,10_b,12_a$&$3_a$\\
\hline
$E_8$&$9_a,12_e,14_a,18_c$
\\
\hline
\caption{Non GIT stable positive rank gradings affording cuspidals}\label{table 1}
\end{longtable}

\item the gradings with $r = 0$ listed in~\autoref{thm:bisup-excep} (we give more detailed account in this case).
\end{enumerate}

Moreover, 
\begin{enumerate}
\item when $r=\dim\fc$, all cuspidal character sheaves have full support.
\item when the grading is $2_a$ for $E_6$, the supporting stratum of the cuspidal character sheaves is $\widecheck{\cO}$, where $\cO$ is the unique nilpotent orbit in $\Lg_{-1}\cong\Lg_1$ such that $(G^\phi)^0$ is of type $G_2$ and such that $\theta|_{(G^\phi)^0}$ gives rise to the split symmetric pair for $G_2$. 
\item when the grading is $3_a$ for $E_7$, the supporting stratum of the cuspidal character sheaves is $\widecheck{\cO}$, where $\cO$ is the unique nilpotent orbit in $\Lg_{-1}$ such that $(G^\phi)^0$ is of type $F_4$ and such that $\theta|_{(G^\phi)^0}$ gives rise to the order 3 GIT stable grading for $F_4$. 

\end{enumerate}
\end{theo}

The proof of~\autoref{thm-excp} is given in~\autoref{sec:exceptional}.

It follows from~\autoref{thm-classical} and~\autoref{thm-excp} (and their proofs) that

\begin{coro}\label{coro:weyl}
 The cuspidal Cartan subspace and the associated Weyl groups are (assuming existence of cuspidal character sheaves)
\begin{enumerate}
\item $\fd\cong\fc$, $W_\fd=W (G_0, \fc)$ when $\dim\Lg_1-\dim\Lg_0=\dim\fc$;
\item $\fd\cong\{0\}$ when $\dim\Lg_1-\dim\Lg_0=0$;
\item when $0<\dim\Lg_1-\dim\Lg_0<\dim\fc$, cuspidal character sheaves occur only in the following cases:
\begin{enumerate}
\item {\bf Type $\bA$} case (b) of~\autoref{thm-classical}: $\dim\fd=1$, $W_\fd\cong G_{m,1,1}$
\item {\bf Type $\BDI$} case (b) of~\autoref{thm-classical}: when $a+b>l$ and $r>0$, $\dim\fd=r$, $W_\fd\cong G_{m,1,r}$

\item {\bf The grading $2_a$ of $E_6$}: $\fd\cong\fc_{t_1,t_2}$  and $W_\fd\cong W(G_2)$ (see~\autoref{cor:E6})
\item {\bf The grading $3_a$ of $E_7$}: $\fd\cong\fc_{w_1}$  and $W_\fd\cong G_5$ (see~\autoref{coro:E7}).
\end{enumerate}
\end{enumerate}
\end{coro}
In the above we have used the Shephard-Todd names for complex reflection groups. In particular, $G_{m,1,r}\cong \mathfrak{S}_r\ltimes(\Z/m\Z)^r$. 

\begin{rema}
The notion of a cuspidal Cartan subspace  depends \textit{a priori} on the supporting stratum  of a cuspidal character sheaf,
but the above theorem shows that it (when exists) does not.  
It would be interesting to find an invariant-theoretic definition of a cuspidal Cartan subspace, independently of the existence of cuspidal character sheaves. A candidate for this is that ``$\fd$ is a maximal abelian subspace of $\fg_1$ consisting of semisimple elements such that there exists a $G_0$-distinguished element $x\in \fg_1$ satisfying $[x,\fd] = 0$''. We have checked that this notion coincides with that of a cuspidal Cartan subspace in almost all cases where cuspidal character sheaves exist. 
\end{rema}

\subsection{Bi-orbital cuspidal character sheaves}In this section we describe the bi-orbital cuspidal character sheaves.

We fix an identification $Z(\SL_n) = \mu_n$, $Z(\Sp_{2n}) = \mu_2$, $Z(\Evii) = \mu_2$ and $Z(\Evi) = \mu_3$.

Assume that $(G, \theta)$ is of classical type: $\bA, \bC, \BD, \iiA$. We further divide the case of type $\BD$ into two subcases, according to whether $\varepsilon\in Z(G)$ acts on $\cC$ as $1$ or $-1$. We denote them by $\BD^{\varepsilon=1}$ and $\BD^{\varepsilon=-1}$. For $n\in\mathbb{N}$ we set
\beq\label{eqn-eta}
    \eta(n) = \begin{cases} 2 & \text{if $n > 0$ and $n\equiv 0,3 \pmod{4}$} \\ 1 & \text{otherwise}\,.\end{cases}
\eeq
\begin{theo}\label{thm:bisup-classical}
Let $(G,\theta)$ be of classical type. Then the bi-orbital cuspidal character sheaves on $\fg_1$ arise exactly when the pair $(G, \theta)$ and the supporting orbit $\cO_{\bs}$ are of the form stated in~\autoref{thm-classical} with $r = 0$. Moreover,  
\begin{enumerate}
\item {\bf Type $\bA$.} $m\nmid N$ and $\bd=\bd_{\bs }$. The bi-orbital cuspidal character sheaves are $\IC(\cO_{\bs },\cE_\psi)$, where $\cE_\psi$ is the local system corresponding to a primitive character $\psi:\mu_N\to \Cc$. 
\item {\bf Type $\iiAI$, $\CI$, $\iiAIII$, $\CIII$, $\BDI^{\varepsilon=1}$, $\BDIII^{\varepsilon=1}$.}   There is a unique local system $\cE$ on $\cO_{\bs }$ such that $\IC(\cO_{\bs },\cE)$ is bi-orbital cuspidal.
\item {\bf Type $\BDI^{\varepsilon=-1}$, $\BDIII^{\varepsilon=-1}$.}  There are $\eta(a)\eta(b)$ (resp. $1+\eta(k)$) local systems $\cE$ on $\cO_{\bs }$ such that $\IC(\cO_{\bs },\cE)$ is bi-orbital cuspidal for $\BDI^{\varepsilon=-1}$ (resp. $\BDIII^{\varepsilon=-1}$).

\end{enumerate}
\end{theo}

In exceptional types $(G,\theta)$ of outer type $\sigma$, we give the table that includes the following data (see~\S\ref{ssec:method})
$$(L,\ \cO_L,\ \#,\ \chi,\ m,\ \on{Kac},\ \text{RLYG label})$$
where 
\begin{itemize}
\item $L$ is a maximal $\sigma$-pseudo-Levi subgroup such that there exists a pair $(\cO_L,\cC_L)$ giving rise to a bi-orbital supercuspidal $\on{IC}(\cC)$ as in~\autoref{ssec:method}
\item $\#$ is the number of such $\cC_L$'s on $\cO_L$
\item $\chi$  indicates the central character in type $\Evi$ and $\Evii$
\item $m=\on{ord}(\theta)$
\item $\on{Kac}$ is the Kac diagram that gives rise to $\theta$
\item In addition, for types $E$, when $\dim\fc>0$, we also indicate the RLYG label of the grading. Note that the only other positive rank grading occurs in type $F_4$ when $m=4$.
\end{itemize}
By~\autoref{lemm:pi'}, these data completely determine the bi-orbital supercuspidal sheaves on simple graded Lie algebras of a fixed outer type. 
Given any distinguished nilpotent $L$-orbit $\cO_L\subseteq \fl$, the weighted Dynkin diagram attaches to $\cO_L$ a weight function $\rho': \Delta_0 \to \{0,2\}$ (where $\Delta_0$ is the set of simple roots); we set $\rho = \frac{1}{2}\rho': \Delta\setminus\{\beta\}\to \{0,1\}$. We will use $\rho$ to label the orbits $\cO_L$.
\begin{theo}\label{thm:bisup-excep}
    Let $(G,\theta)$ be of exceptional  type with outer type $\sigma$. The bi-orbital cuspidal sheaves are the $\Ind^{\fg_1}_{\fp_1}\IC(\cC_L|_{\fl_1\cap \cO_L})$ arising from the data tabulated below:

\begingroup
\newenvironment{cptmatrix}{\begingroup\renewcommand*{\arraystretch}{.7}\setlength\arraycolsep{1pt}\begin{matrix}} {\end{matrix}\endgroup}
\setlength\cellspacetoplimit{2pt}
\setlength\cellspacebottomlimit{1pt}
\pgfkeys{/Dynkin diagram,
edge length=.3cm}

\begin{center}
	\newcommand{\wdynkin}[3]{$\begin{cptmatrix}#1 & #2 & #3\end{cptmatrix}$}
	\begin{tabular}{|Sc|Sc|Sc|Sc|Sc|}
		\multicolumn{5}{c}{Type $\Gii$, $\dynkin[affine mark=*, ordering=Kac, labels={1,2,3}] G[1]2$} \\
		\hline
		$L$ & $\cO_L$ & $\#$& $m\;(k \ge 2)$ & Kac  \\
		\hline
		\multirow{2}{*}{$\Gii$} & \multirow{2}{*}{\wdynkin{\emptyset}{1}{0}} &\multirow{2}*{$1$}& $1$ & \wdynkin{1}{0}{0}  \\
		\cline{4-5}
		& & & $k + 2$ & \wdynkin{k}{1}{0}  \\
		\hline
		\multirow{2}{*}{$\frac{\SL_2\times\SL_2 }{ (\zeta_2,\zeta_2)}$} & \multirow{2}{*}{\wdynkin{1}{\emptyset}{1}} &\multirow{2}*{$1$}& $4$ & \wdynkin{1}{0}{1}  \\
		\cline{4-5}
		& & & $2k + 4$ & \wdynkin{1}{k}{1}  \\
		\hline
		$\SL_3$ & \wdynkin{1}{1}{\emptyset} & $2$& $3k + 3$ & \wdynkin{1}{1}{k}  \\
		\hline
	\end{tabular}
\end{center}

\begin{center}
	\newcommand{\wdynkin}[3]{$\begin{cptmatrix}#1 & #2 & #3\end{cptmatrix}$}
	\begin{tabular}{|Sc|Sc|Sc|Sc|Sc|}
		\multicolumn{5}{c}{Type $\iiiDiv$, $\dynkin[affine mark=*, labels={1,2,1}] D[3]{**}$} \\
		\hline
		$L$ & $\cO_L$ & $\#$& $m/3\;(k \ge 2)$ & Kac  \\
		\hline
		$\Gii$ & \wdynkin{\emptyset}{0}{1} & $1$& $k + 1$ & \wdynkin{k}{0}{1}  \\
		\hline
		$\frac{\SL_2\times\SL_2 }{ (\zeta_2,\zeta_2)}$ & \wdynkin{1}{\emptyset}{1} & $1$& $2k + 2$ & \wdynkin{1}{k}{1}  \\
		\hline
	\end{tabular}
\end{center}

\begin{center}
	\newcommand{\wdynkin}[5]{$\begin{cptmatrix}#1 & #2 & #3 & #4 & #5\end{cptmatrix}$}
	\begin{tabular}{|Sc|Sc|Sc|Sc|Sc|}
		\multicolumn{5}{c}{Type $\Fiv$, $\dynkin[affine mark=*, ordering=Kac, labels={1,2,3,4,2}] F[1]4$} \\
		\hline
		$L$ & $\cO_L$ & $\#$ & $m\;(k \ge 2)$ & Kac \\
		\hline
		\multirow{2}{*}{$\Fiv$} & \multirow{2}{*}{\wdynkin{\emptyset}{0}{1}{0}{0}} &\multirow{2}*{$1$}& $1$ & \wdynkin{1}{0}{0}{0}{0} \\
		\cline{4-5}
		& & & $k + 3$ & \wdynkin{k}{0}{1}{0}{0} \\
		\hline
		\multirow{2}{*}{$\frac{\SL_2 \times \Sp_6 }{ (\zeta_2, \zeta_2)}$} & \multirow{2}*{\wdynkin{1}{\emptyset}{1}{0}{1}} &\multirow{2}*{$1$}& $4$ & \wdynkin{0}{1}{0}{0}{1} \\
		\cline{4-5}
		& & & $2k + 6$ & \wdynkin{1}{k}{1}{0}{1} \\
		\hline
		\multirow{2}{*}{$\frac{\SL_3 \times \SL_3 }{ (\zeta_3, \zeta_3)}$} & \multirow{2}{*}{\wdynkin{1}{1}{\emptyset}{1}{1}} &\multirow{2}*{$2$} & $9$ & \wdynkin{1}{1}{0}{1}{1} \\
		\cline{4-5}
		& & & $3k + 9$ & \wdynkin{1}{1}{k}{1}{1} \\
		\hline
		\rule{0pt}{2.3ex} 
		$\frac{\SL_4 \times \SL_2 }{ (\zeta_2, \zeta_2)}$ & \wdynkin{1}{1}{1}{\emptyset}{1} & $2$& $4k + 8$ & \wdynkin{1}{1}{1}{k}{1}  \\
		\hline
		$\Spin_9$ & \wdynkin{1}{0}{1}{0}{\emptyset} & $1$& $2k + 4$ & \wdynkin{1}{0}{1}{0}{k}  \\
		\hline
	\end{tabular}
\end{center}

\begin{center}
	\newcommand{\wdynkin}[5]{$\begin{cptmatrix}#1 & #2 & #3 & #4 & #5\end{cptmatrix}$}
	\begin{tabular}{|Sc|Sc|Sc|Sc|Sc|}
		\multicolumn{5}{c}{Type $\iiEvi$, \dynkin[affine mark=*, labels={1,2,3,2,1}] E[2]{****}} \\
		\hline
		$L$ & $\cO_L$ & $\#$& $m/2\;(k \ge 2)$ & Kac  \\
		\hline
		$\Fiv$ & \wdynkin{\emptyset}{0}{0}{1}{0} & $1$& $k + 2$ & \wdynkin{k}{0}{0}{1}{0}  \\
		\hline
		$\frac{\SL_3\times\SL_3 }{ (\zeta_3, \zeta_3)}$ & \wdynkin{1}{1}{\emptyset}{1}{1} & $2$& $3k + 6$ & \wdynkin{1}{1}{k}{1}{1}  \\
		\hline
		
	\end{tabular}
\end{center}

\begin{center}
	\newcommand{\wdynkin}[7]{$\begin{cptmatrix} & & #1 & & \\ & & #2 & & \\ #3 & #4 & #5 & #6 & #7 \end{cptmatrix}$}
	\begin{tabular}{|Sc|Sc|Sc|Sc|Sc|Sc|Sc|}
		\multicolumn{7}{c}{Type $\Evi$, $\dynkin[affine mark=*, ordering=Kac, labels={1,1,2,3,2,1,2}] E[1]{******}$, \quad up to an order 3 diagram rotation} \\
		\hline
		$L$ & $\cO_L$  & $\#$ & $\chi$ & $m\;(k \ge 2)$ & Kac& RLYG label\\
		\hline
		\multirow{7}*{$\Evi$} & \multirow{7}*{\wdynkin{\emptyset}{0}{1}{0}{1}{0}{1}} & \multirow{7}*{$2$} & \multirow{7}*{$\neq 1$} & $1$ & \wdynkin{1}{0}{0}{0}{0}{0}{0}&\\
		\cline{5-7}
		& &  & & $4$ & \wdynkin{0}{1}{1}{0}{0}{0}{1}&$4_b$\\
		\cline{5-7}
		& &  &  & $5$ & \wdynkin{0}{0}{1}{0}{1}{0}{1}&$5_a$\\
		\cline{5-7}
		& &  & & $k + 5$ & \wdynkin{k}{0}{1}{0}{1}{0}{1}&\\
		\hline
		\multirow{7}*{$\frac{\SL_2\times \SL_6}{(\zeta_2, \zeta_2)}$} & \multirow{7}*{\wdynkin{1}{\emptyset}{1}{1}{1}{1}{1}} & \multirow{7}*{$2$}  & \multirow{7}*{$1$} & $4$ & \wdynkin{1}{0}{0}{0}{1}{0}{0}&\\
		\cline{5-7}
		&  &  & & $8$ & \wdynkin{0}{1}{1}{1}{0}{1}{1}&$4_a$\\
		\cline{5-7}
		&  & & & $10$ & \wdynkin{1}{0}{1}{1}{1}{1}{1}&$8_b$\\
		\cline{5-7}
		&  & & & $2k + 10$ & \wdynkin{1}{k}{1}{1}{1}{1}{1}&\\
		\hline
		$\frac{\SL_3\times\SL_3\times\SL_3}{(\zeta_3,\zeta_3,\zeta_3)}$ & \wdynkin{1}{1}{1}{1}{\emptyset}{1}{1}  & $2$ & $1$ & $3k + 9$ & \wdynkin{1}{1}{1}{1}{k}{1}{1}&\\
		\hline
	\end{tabular}
\end{center}

\begin{center}
	\newcommand{\wdynkin}[8]{$\begin{cptmatrix} & & & #1 & & & \\ #2 & #3 & #4 & #5 & #6 & #7 & #8 \end{cptmatrix}$}
	\begin{tabular}{|Sc|Sc|Sc|Sc|Sc|Sc|Sc|}
		\multicolumn{7}{c}{Type $\Evii$, $\dynkin[affine mark=*, ordering=Kac, labels={1,2,3,4,3,2,1,2}] E[1]7$, \quad up to the diagram reflection} \\
		\hline
		$L$ & $\cO_L$ & $\#$ & $\chi$ & $m\;(k \ge 2)$ & Kac&RLYG label \\
		\hline
		\multirow{4}*{$\Evii$} & \multirow{4}*{\wdynkin{0}{\emptyset}{0}{0}{1}{0}{0}{1}} & \multirow{4}*{$1$} & \multirow{4}*{$\neq 1$} & $1$ & \wdynkin{0}{1}{0}{0}{0}{0}{0}{0}&\\
		\cline{5-7}
		&  &  & & $5$ & \wdynkin{0}{0}{0}{0}{1}{0}{0}{1}&$5_a$\\
		\cline{5-7}
		 &  &  & & $k + 5$ & \wdynkin{0}{k}{0}{0}{1}{0}{0}{1}&\\
		\hline
		\multirow{4}*{$\frac{\SL_3\times \SL_6}{(\zeta_3, \zeta_3)}$} & \multirow{4}*{\wdynkin{1}{1}{1}{\emptyset}{1}{1}{1}{1}}  & \multirow{4}*{$2$} & \multirow{4}*{$1$} & $9$ & \wdynkin{0}{0}{1}{0}{1}{0}{1}{1}&$9_a$\\
		\cline{5-7}
		& & &  &  $15$ & \wdynkin{1}{1}{1}{0}{1}{1}{1}{1}&\\
		\cline{5-7}
		& &  &  & $3k + 15$ & \wdynkin{1}{1}{1}{k}{1}{1}{1}{1}&\\
		\hline
		$\frac{\SL_4\times\SL_4\times\SL_2}{(\zeta_4,\zeta_4,\zeta_2)}$ & \wdynkin{1}{1}{1}{1}{\emptyset}{1}{1}{1}  & $2$ & $1$ & $4k + 14$ & \wdynkin{1}{1}{1}{1}{k}{1}{1}{1}&\\
		\hline
	\end{tabular}
\end{center}

\begin{center}
	\newcommand{\wdynkin}[9]{$\begin{cptmatrix} & & #1 & & & & & \\ #2 & #3 & #4 & #5 & #6 & #7 & #8 & #9 \end{cptmatrix}$}
	\begin{tabular}{|Sc|Sc|Sc|Sc|Sc|Sc|}
		\multicolumn{5}{c}{Type $\Eviii$, $\dynkin[affine mark=*, ordering=Kac, labels={1,2,3,4,5,6,4,2,3}] E[1]{********}$} \\
		\hline
		$L$ & $\cO_L$ & $\#$ & $m\;(k \ge 2)$ & Kac&RLYG label \\
		\hline
		\multirow{3}*{$\Eviii$} & \multirow{3}*{\wdynkin{0}{0}{0}{0}{1}{0}{0}{0}{\emptyset}} &\multirow{3}*{$1$}& $1$ & \wdynkin{0}{0}{0}{0}{0}{0}{0}{0}{1}& \\
		\cline{4-6}
		& & & $k + 5$ & \wdynkin{0}{0}{0}{0}{1}{0}{0}{0}{k} &\\
		\hline
		\multirow{3}*{$\frac{\Evii\times \SL_2}{(\zeta_2, \zeta_2)}$} & \multirow{3}*{\wdynkin{0}{0}{0}{1}{0}{0}{1}{\emptyset}{1}} &\multirow{3}*{$1$} & $4$ & \wdynkin{1}{0}{0}{0}{0}{0}{0}{0}{1}& $4_b$\\
		\cline{4-6}
		& & & $2k + 10$ & \wdynkin{0}{0}{0}{1}{0}{0}{1}{k}{1} &\\
		\hline
		\multirow{3}*{$\frac{\Evi\times \SL_3}{(\zeta_3, \zeta_3)}$} & \multirow{3}*{\wdynkin{0}{1}{0}{1}{0}{1}{\emptyset}{1}{1}} &\multirow{3}*{$2$}& $9$ & \wdynkin{0}{1}{0}{0}{1}{0}{0}{1}{0}&$9_c$ \\
		\cline{4-6}
		& & & $3k + 15$ & \wdynkin{0}{1}{0}{1}{0}{1}{k}{1}{1} &\\
		\hline
		\multirow{3}*{$\frac{\SL_5\times \SL_5 }{ (\zeta_5, \zeta_5^2)}$} & \multirow{3}*{\wdynkin{1}{1}{1}{1}{\emptyset}{1}{1}{1}{1}} &\multirow{3}*{$4$} & $25$ & \wdynkin{1}{1}{1}{1}{0}{1}{1}{1}{1} &\\
		\cline{4-6}
		& & & $5k + 25$ & \wdynkin{1}{1}{1}{1}{k}{1}{1}{1}{1} &\\
		\hline
		$\frac{\SL_6\times\SL_3\times\SL_2 }{ (\zeta_6,\zeta_3,\zeta_2)}$ & \wdynkin{1}{1}{1}{\emptyset}{1}{1}{1}{1}{1} & $2$ & $6k + 24$ & \wdynkin{1}{1}{1}{k}{1}{1}{1}{1}{1}& \\
		\hline
		\multirow{3}*{$\SO_{16}$} & \multirow{3}*{\wdynkin{0}{\emptyset}{0}{1}{0}{0}{1}{0}{1}} &\multirow{3}*{$1$}& $8$ & \wdynkin{0}{0}{1}{0}{0}{0}{1}{0}{1} &$8_c$\\
		\cline{4-6}
		& & & $2k + 10$ & \wdynkin{0}{k}{0}{1}{0}{0}{1}{0}{1}& \\
		\hline
	\end{tabular}
    
\end{center}

\endgroup
\end{theo}

The proofs of~\autoref{thm:bisup-classical} and~\autoref{thm:bisup-excep} are given in~\S\, \ref{sec-bicus}.

\subsection{Uniqueness of support of cuspidal character sheaves}As a consequence of our classification results, we have the following corollary, which says that the support of cuspidal character sheaves is unique when the grading and the central character are fixed. 
\begin{coro}
\label{theo:support}

\begin{enumerate}
\item If $\dim\Lg_1-\dim\Lg_0=\dim\fc>0$, then all cuspidal character sheaves have full support.
\item
    If $(G, \theta)$ is not of type $\mathbf{BDIII}$, then every cuspidal character sheaf on $\fg_1$ has the same support. 
    \item 
    Assume that $(G, \theta)$ is of type $\mathbf{BDIII}$, so that $G \cong \Spin(N)$.  
    Every cuspidal character sheaf on $\fg_1$ with trivial (resp. non-trivial) $\varepsilon$-action has the same support, closure of a $G_0$-distinguished nilpotent orbit. 
\end{enumerate}
    
\end{coro}
\begin{proof}
Note that $r=\dim\fc = 1$ in type $\bA$ when $m=N = \dim V$. Thus part (1) follows from~\autoref{thm-classical} and~\autoref{thm-excp}. Note that in case (3) of~\autoref{thm-classical}, the gradings in case (a) and those in case (b) do not coincide. So part (2) follows from~\autoref{thm-classical},~\autoref{thm-excp} and~\autoref{thm:bisup-excep}. Part (3) is clear from case (5) of~\autoref{thm-classical}. 
\end{proof}

\section{Distinguished elements in classical types}\label{sec:distinguished}

In this section we describe the $G_0$-distinguished elements in $\fg_1$ when $(G, \theta)$ is of classical type.

\begin{prop}\label{prop:distinguished-II/III}
\begin{enumerate}
\item If $(G, \theta)$ is of type $\mathbf{A}$, then there is a $G_0$-distinguished element in $\Lg_1$ if and only if the grading is as in Theorem~\ref{thm-classical} (1).
 
\item If $(G, \theta)$ is of type $\iiAII, \BDII, \CII$, then there are no $G_0$-distinguished elements in $\fg_1$. 
\item If $(G, \theta)$ is of type $\iiAIII, \BDIII, \CIII$, then every $G_0$-distinguished element is nilpotent. 
\end{enumerate}
\end{prop}

For any set $S$, let $\Pf(S)$ denote the set of finite subsets of $S$. 
  
If $(G,\theta)$ is of type $\iiAI$, $\CI$ or $\BDI$, we define a map as follows:
\begin{align}\label{multi-I}
			\Pf(\N)^2 &\to \{\text{multi-segments}\} \nonumber\\
		(A, B) &\mapsto \bs_{(A,B)}:= \begin{cases} \displaystyle{\sum_{a\in A} [-a, a] + \sum_{b\in B} [l-b,l+ b]} & \BDI \\ \displaystyle{\sum_{a\in A} [-(a+\tfrac{1}{2}), a+\tfrac{1}{2}] + \sum_{b\in B} [l -(b+\tfrac{1}{2}), l + b+\tfrac{1}{2}]} & \CI \\ \displaystyle{\sum_{a\in A} [-a, a] + \sum_{b\in B} [l-b,l+ 1+ b]} & \iiAI\,. \end{cases}
    \end{align}

For any vector $\bd = (d_i)_{i\in I}\in \Z^I$ satisfying $d_i = d_{-i}$, we define
\begin{equation}\label{eq:F}
	F(\bd) = \begin{cases} \displaystyle\frac{1}{2}\left(d_0 + d_{l}-\sum_{k = 0}^{l - 1}(d_k - d_{k+1})^2\right) & \BDI, \iiAI \\ \displaystyle\frac{1}{2}\left(d_{\frac{1}{2}} + d_{l - \frac{1}{2}}-\sum_{k = 1}^{l - 1}(d_{k - \frac{1}{2}} - d_{k + \frac{1}{2}})^2\right) & \CI.  \end{cases}
\end{equation}

Let $x=x_s+x_n\in\Lg_1$. Let us write $\fh_x=Z_\fg(x_s)_{\on{der}}$, $\fh_{x,i}=\fh_x\cap\fg_i$, and $\fZ_{x,1}=Z(Z_\fg(x_s))_1$.

\begin{prop}\label{prop:distinguished-I}
	Assume that $\fg(\bd)_*$ lies in the families $\BDI, \CI, \iiAI$ and $\fg(\bd)_1$ admits a $G(\bd)_0$-distinguished element. Let $r = \dim\fg_1 - \dim \fg_0$. Then, we have:
\begin{enumerate}
\item
The dimension vector $\bd = \dim_I V$ is of the form $\bd = \bd_{\bs_{A,B} }+ r\mathbf{1}$, where $(A,B)\in \Pf(\N)^2$ satisfies $F(\bd_{\bs_{A,B} }) = 0$ and $\mathbf{1} = (1)_{k\in I}\in \N^I$. 
\item
	If $x \in \fg(\bd)_1$ is a $G(\bd)_0$-distinguished element, then $\dim \fZ_{x,1} \le r$. 
\item 
	There exists a $G_0$-distinguished element $x = x_s + x_n\in \fg(\bd)_1$ such that $\dim \fZ_{x,1} = r$. Moreover,  
    $\dim x_sV=r\mathbf{1}$, $\dim\fh_{x,0}=\dim\fh_{x,1}$
    and $\fh_{x}$ has a unique distinguished nilpotent orbit $\cO_{\bs_{(A,B)}}$. 
\end{enumerate}
\end{prop}

\subsection{Type \texorpdfstring{$\bA$}{A} and proof of~\autoref{prop:distinguished-II/III} part (1)}
Assume that $m \ge 2$. 
Suppose that $x = x_s + x_n\in \fg_1$ is distinguished.  Let $\Lambda$ be the set of eigenvalues of $x_s\in \End(V)$ and let $V(\lambda)$ be the $\lambda$-eigenspace of $x_s$ in $V$ for $\lambda\in \Lambda$. 
\begin{lemm}
    The set $\Lambda$ forms a single $\mu_m$-orbit in $\C$.
\end{lemm}
\begin{proof}
    The automorphism $\theta\in \GL(V)$ induces an isomorphism $V(\zeta\lambda) \cong V(\lambda)$ for each $k\in \Z/m$ and $\lambda\in \Lambda$, so $\Lambda$ is $\mu_m$-stable under multiplication. Suppose there is a decomposition $\Lambda = \Lambda' \sqcup \Lambda''$ into non-empty $\mu_m$-stable subsets, then $V' = \bigoplus_{\lambda\in \Lambda'}V(\lambda)$ and $V'' = \bigoplus_{\lambda\in \Lambda''}V(\lambda)$ are preserved by $\theta$, $x_s$ and $x_n$, and $V = V'\oplus V''$. It follows that the cocharacter
    \[
    \rho:\Cc\to \SL(V),\quad \rho(t)|_{V'} = t^{\dim V''}\id_{V'},\quad \rho(t)|_{V''} = t^{-\dim V'}\id_{V''}
    \]
    is non-trivial and its image centralises $\theta$, $x_s$ and $x_n$. It follows that $\rho(\Cc)$ is a non-trivial torus lying in $Z(Z_G(x))_0$, contradicting the distinguishedness of $x$. It follows that $\Lambda$ is a single $\mu_m$ orbit. 
\end{proof}
We have either $\Lambda = \mu_m\lambda_0$ for some $\lambda_0\in \Cc$ or $\Lambda = \{0\}$. In the former case, $x_s$ is an isomorphism and $x_s^m = \lambda_0^m\id_V$. The endomorphism $\varphi = x_s^{-1}x_n|_{V_0}\in \End(V_0)$ is nilpotent. We claim that $\varphi$ is cyclic. Indeed, if $V_0 = V_0'\oplus V_0''$ is a decomposition into $\varphi$-stable non-zero subspaces, then $V_k = V_k'\oplus V_k''$ with $V_k' = x_s^k V_0'$ and $V_k'' = x_s^k V_0''$; put $V' = \bigoplus_{k\in \Z/m}V'_k$ and $V'' = \bigoplus_{k\in \Z/m}V''_k$; then the cocharacter
\[
    \rho:\Cc\to \SL(V),\quad \rho(t)|_{V'} = t^{\dim V''}\id_{V'},\quad \rho(t)|_{V''} = t^{-\dim V'}\id_{V''}
\]
is non-trivial and its image centralises $\theta$, $x_s$ and $x_n$, which contradicts the distinguishedness of $x$. \par
In the latter case, we have $x_s = 0$, so $x = x_n$ is nilpotent. It is easy to show that $x\in \End(V)$ is cyclic, whence regular nilpotent. \par

Therefore, the existence of distinguished element implies that the $\Z/m$-grading on $V$ is in either of the following cases: 
\begin{enumerate}
\item $\dim\fg_1 - \dim\fg_0 = 1$, and $\dim V_k = d$ for every $k\in \Z/m$;
\item $\dim \fg_1 - \dim \fg_0 = 0$, and there exists integers $a < b$ with $b < a + m$ such that $\dim V_k = d + 1$ for $k\in [a, b-1]$ and $\dim V_k = d$ for $k\in [b, a + n - 1]$.
\end{enumerate}
\autoref{prop:distinguished-II/III} (1) follows.

\subsection{Type \texorpdfstring{$\BD$}{BD}, \texorpdfstring{$\bC$}{C} and \texorpdfstring{$\iiA$}{A (unitary)}}
Consider $(G, \theta)$ of type $\BD, \bC$ and $\iiA$.

\begin{lemm}\label{lem:dist-multi-seg}
The distinguished nilpotent $G_0$-orbits in $\fg_1$ are in bijective correspondence with the multi-segments $\bs $ on $I$ satisfying the following conditions:
\begin{enumerate}
\item (self-duality) If $[a,b]\in \bs $, then $[a,b] = [a,b]^*$.
\item (multiplicity-freeness) Every segment $[a,b]\in \bs $ occurs with multiplicity at most $1$. 
\item (parity) Every segment $[a,b]\in \bs $ satisfies 
\[
(-1)^{b - a} = \begin{cases} 1 & \BD \\ -1 & \bC \\ (-1)^{2(a+b)/m} & \iiA\end{cases}
\]
\end{enumerate}
\end{lemm}
\begin{proof}
The proof is completely analogous to the ungraded case in~\cite[Propositions 3.6--3.8]{BC76}.
Let $x\in \fg_1^{\nil}$. Then, $x$ is distinguished if and only if $M\cap M^{\perp} = 0$ for every $I$-graded $\C[x]$-module constituent $M\subseteq V$. We deduce from this that if $\bs  = \sum c_{[a, b]}[a, b]$ is the multi-segment of $x$, then $c_{[a, b]} = 0$ whenever $[a,b]^* \neq [a, b]$ and $c_{[a, b]} \le 1$ whenever $[a, b]^* = [a, b]$. \autoref{lem:multi-seg} implies that $\bs $ satisfies the required conditions. \par
Now, suppose that $x, x'\in \fg_1^{\nil}$ are distinguished nilpotent elements labelled by the same multi-segment. Then, we can find $\tau\in \Aut(V)$ preserving the bilinear form $(, )$ and such that $\tau(V_i) = V_i$ for $i\in I$ and $\tau x = x'\tau$. In the cases of types $\iiA$ and $\bC$, we have $\tau\in G_0$, so $G_0x = G_0x'$. Suppose $(G, \sigma)$ is of type $\BD$. Then, we have $\tau\in \on{O}(V)$. Pick any indecomposable $I$-graded $x$-stable direct factor $M\subseteq V$ and define $\tau'\in \SO(V)$ by $\tau'|_{M} = (\det \tau)\tau|_{M}$ and $\tau'|_{M^{\perp}} = \tau|_{M^{\perp}}$. Since $\dim M$ is odd by the distinguishedness of $x$, we have $\tau'\in \SO(V)$. It follows that any lifing $\tau''\in G = \Spin(V)$ of $\tau'$ lies in $G_0$ and satisfies $\tau''x = x'\tau''$, so $G_0x = G_0x'$. 
\end{proof}

\subsection{Proof of~\autoref{prop:distinguished-II/III} parts (2) and (3)}
Assume first that $(G, \theta)$ is of type II. Let $x\in \fg_1$. We view $V$ as an $I$-graded $\C[x]$-module. 
\begin{lemm}
	The bilinear form on $V$ induces an isomorphism of $I$-graded $\C[x]$-modules:
	\[
		V\xrightarrow{\sim} V^*,\quad V_i \ni v\mapsto \begin{cases} (-1)^i\langle v, \relbar\rangle  & \CII \\ (-1)^{i + \frac{1}{2}}\langle v, \relbar\rangle   & \BDII \\ \zeta_{m_0}^{(i + \frac{1}{2})(l + 1)}\langle v, \relbar\rangle & \iiAII\end{cases}
	\]
\end{lemm}
\begin{lemm}
	Every indecomposable direct summand of $V$, viewed as $I$-graded $\C[x]$-module, is isotropic.
\end{lemm}
\begin{proof}
Let $M\subseteq V$ be an indecomposable direct summand. By Fitting's lemma, the action of $x$ on $M$ is either nilpotent or an automorphism. In the former case, we can complete $x$ into an $\fsl_2$-triple $(x, h, f)\in \End(V)^3$ of linear operators such that $h(M_i)\subseteq M_i$ and $f(M_i)\subseteq M_{i-1}$ for every $i\in I$. Then, $M^f = \{v\in M\mid fm = 0\}$ is one dimensional due to the indecomposability of $M$. Thus, any non-zero element $0\neq m\in M^f$ is a homogeneous cyclic generator of $M$ as $\C[x]$-module. In the case where $x$ is an automorphism, the indecomposability of $M$ implies that of $M_0$ as $\C[x^m]$-module. Any cyclic generator of $M_0$ as $\C[x^m]$-module is a cyclic generator of $M$ as $\C[x]$-module. It follows that $M$ admits a homogenous cyclic generator $v\in M_i$ for some $i\in I$ in both cases. To complete the proof, it remains to show that $\langle x^rv, x^sv\rangle = 0$ for $r,s\in \N$. \par

We know that $\langle x^rv, x^sv\rangle = 0$ unless $m_0\mid (r+s + 2i)$. Note that $2i\in \Z/m_0 \Z$ is odd (resp. even) for types $\BDII$ (resp. $\CII$). Suppose that $m_0\mid (r+s + 2i)$. Then, 
	\[
	\langle x^rv, x^sv\rangle  = \begin{cases}(-1)^{r+s}\langle x^sv, x^rv\rangle =  (-1)^{r+s+1}\langle x^rv, x^sv\rangle = -\langle x^rv, x^sv\rangle & \CII \\ (-1)^{r+s}\langle x^sv, x^rv\rangle = (-1)^{r+s}\langle x^rv, x^sv\rangle =  -\langle x^rv, x^sv\rangle & \BDII \\ (-\zeta_{m})^{r - s}\langle x^sv, x^rv\rangle = (-\zeta_{m})^{r - s}\langle x^rv, \gamma x^sv\rangle\\ \quad = (-1)^{r+s}\zeta_{m}^{r + s + 2i}\langle x^rv, x^sv\rangle & \iiAII\end{cases}
	\]
	In the case of $\iiAII$, if $m \mid (r+s + 2i)$, then $\zeta_{m}^{r + s + 2i} = 1$ and $r+s$ is odd; if $m \nmid (r+s + 2i)$, then $\zeta_{m}^{r + s + 2i} = -1$ and $r+s$ is even; therefore, we have $(-1)^{r+s}\zeta_{m}^{r + s + 2i} = -1$ in either case. It follows that $\langle x^rv, x^sv\rangle = 0$ in all cases. This completes the proof.
\end{proof}

\begin{proof}[Proof of~\autoref{prop:distinguished-II/III} (2)]
	Let $x \in \fg_1$. We view $V$ as an $I$-graded $\C[x]$-module. Let $M\subseteq V$ be an indecomposable direct summand. Then, $M$ is isotropic. On the other hand, we have $V\cong V^*$, so we can find a direct summand $M'\subseteq V$ such that the bilinear form on $V$ restricts to a perfect pairing $M\times M'\to \C$. Since $M$ is isotropic, we have $M\cap M' = 0$. Then, the linear operator 
	\[
		\varphi: V\to V,\quad \varphi|_{M} = \id_M,\quad \varphi|_{M'} = -\id_{M'},\quad \varphi|_{(M \oplus M')^{\perp}} = 0
	\]
	is semisimple and lies in $\fg_0$. Moreover, we have $[\varphi,x] = 0$. It follows that $x$ is not distinguished.
\end{proof}

Suppose now that $(G, \theta)$ is of type III.
\begin{proof}[Proof of~\autoref{prop:distinguished-II/III} (3)]
	Suppose that $x = x_s + x_n\in \fg_1$ is $G_0$-distinguished. Then, $x_s^{m_0}\in \End(V)$ lies in $\fg_0$ and commutes with $x_n$. The $G_0$-distinguishedness of $x$ implies that $x_s = 0$. 
\end{proof}

\subsection{Proof of~\autoref{prop:distinguished-I}}
 We first describe the distinguished nilpotent orbits in the families $\mathbf{BDI}, \mathbf{CI}$ and $\li^2{\mathbf{AI}}$. Assume $(G, \theta)$ belongs to one of these families. It follows from~\autoref{lem:dist-multi-seg} that
\begin{lemm} \label{lem:distinguished-orbits} The map in~\eqref{multi-I} induces a bijection
    \begin{align*}
			\Pf(\N)^2 &\to \bigsqcup_{\bd}\{\text{$G(\bd)_0$-distinguished nilpotent orbits in $\fg(\bd)_1$}\},\ 
		(A, B) \mapsto \cO_{\bs_{(A,B)}}\,.
    \end{align*}
\end{lemm}

Define the function
\[
	 \Pf(\N)^2\to \N^{I},\quad (A, B)\mapsto \bd_{A, B} = (d_i)_{i\in I}
\]
by
\[
	\sum_{i\in I} d_it^i = \begin{cases} \displaystyle\sum_{a\in A} t^{-a}\frac{1 - t^{2a+1}}{1 - t} + \sum_{b\in B} t^{l-b}\frac{1 - t^{2b+1}}{1 - t} & \BDI \\ \displaystyle\sum_{a\in A} t^{-a-\frac{1}{2}}\frac{1 - t^{2a+2}}{1 - t} + \sum_{b\in B} t^{l-b-\frac{1}{2}}\frac{1 - t^{2b+2}}{1 - t} & \CI \\ \displaystyle\sum_{a\in A} t^{-a}\frac{1 - t^{2a+1}}{1 - t} + \sum_{b\in B} t^{l-b}\frac{1 - t^{2b+2}}{1 - t}  & \iiAI.\end{cases}
\]
Then, the distinguished nilpotent orbit $\cO_{\bs_{(A,B)}}$ lies in the Lie algebra $\fg(\bd_{A,B})_1$. \par

Recall the function $F$ defined in~\eqref{eq:F}. 
We have $F(\bd_{A,B}) = \dim\fg_1 - \dim\fg_0$. 
In order to study the value $F(\bd_{A,B})$ for pairs $(A, B)\in \Pf(\N)^2$, we introduce a new parametrisation. Consider the following bijection:
    \begin{align*}
        \{0, 1\}\times \N &\to I\times\N \\
				(0,a) & \mapsto \begin{cases}(a+ m_0\Z, \lfloor \frac{a}{l}\rfloor) & \BDI, \iiAI \\ (a+\frac{1}{2}+ m_0\Z, \lfloor \frac{a}{l}\rfloor) &  \CI\end{cases}\\
				(1,b) &\mapsto \begin{cases}(b+l+1+m_0\Z, \lfloor \frac{2b+1}{2l}\rfloor)& \iiAI \\ (b+l+\frac{1}{2}+m_0\Z, \lfloor \frac{b}{l}\rfloor)& \CI \\ (b+l+m_0\Z, \lfloor \frac{b}{l}\rfloor) & \BDI\end{cases}
    \end{align*}
    which induces a bijection $\beta: \Pf(\N)^2\to \Pf(I\times\N) = \Pf(\N)^I$. 
		\begin{lemm}\label{lemm:dist-orb-classical} Let $(A,B)\in \Pf(\N)^2$ and denote $\beta(A,B) = (S_k)_{k\in I}\in \Pf(\N)^{I}$. Then, we have:
			\begin{enumerate}
				\item
					$F(\bd_{A,B})\ge 0$.
				\item
					$F(\bd_{A, B}) = 0$ precisely when the following conditions hold for each $k\in I$:
					\begin{enumerate}
						\item
							$S_k = [n_k] := \{0, 1, \ldots, n_k-1\}$ for some $n_k\in \N$;
						\item
							$S_k\neq \emptyset$ implies $S_{-k-1} = \emptyset$.
					\end{enumerate}
				\item
					There exists a unique $(A_0, B_0)\in \Pf(\N)^2$ such that $F(\bd_{A_0, B_0}) = 0$ and $\bd_{A, B} - \bd_{A_0, B_0} = r \mathbf{1}$, where $r\in \Z$. Moreover, we have $r = F(\bd_{A,B})$. 
    \end{enumerate}
\end{lemm}
\begin{proof}
	Let us give the proof for type $\BDI$ and leave the other two cases to the reader. In this case, we have $I = \Z/m\Z$ with $m = 2l$. If we write $\bd_{A, B} = (d_i)_{i\in I}$, then
    \[
    d_0 + d_{l} = \sum_{a\in A}(2\lfloor a/l \rfloor + 1) + \sum_{b\in B}(2\lfloor b/l \rfloor + 1).
    \]
    For each $k\in \Z / m\Z$, set 
    \[
    A_k = \{a\in A\mid a \equiv k \pmod{m} \},\quad B_k = \{b\in B\mid b \equiv k \pmod{m}\}.
    \]
    Then we have for $0 \le k \le l - 1$:
    \[
        d_k - d_{k+1} = \#A_k + \#B_{k+l} - \#A_{-k-1} - \#B_{-k + l - 1}= \# S_k - \# S_{-k-1}.
    \]
    for $0\le k\le l - 1$ and
    \begin{align*}
    d_0 + d_{l} &= \sum_{k\in \Z/m\Z}\sum_{n\in S_k}(2n + 1) \ge \sum_{k\in \Z/m\Z}(\# S_k)^2 \\
    &\ge\sum_{k = 0}^{l - 1}(\# S_k - \# S_{-k-1})^2 = \sum_{k = 0}^{l - 1}(d_k - d_{k+1})^2.
    \end{align*}
		From these inequalities and~\eqref{eq:F}, we conclude that $F(\bd_{A, B}) \ge 0$ and the equality holds precisely under the stated conditions. Finally, for $k\in I$, we set
		\[
			S'_k = \begin{cases} [d_k - d_{k+1}] & \text{if $d_k > d_{k+1}$}; \\ \emptyset & \text{otherwise.}\end{cases}
		\]
		Then, the pair $(A_0, B_0) = \beta^{-1}( (S'_k)_{k\in I})\in \Pf(\N)^2$ satisfies the requirement. 
\end{proof}

\begin{proof}[Proof of~\autoref{prop:distinguished-I}]
	  
    Let $x = x_s + x_n\in \fg(\bd)_1$ be a distinguished element. Since $x_s$ operates semisimply on $V$, it induces an orthogonal decomposition $V = V'\oplus V''$, where $V' = x_s V$ and $V'' = \ker x_s$, and an isomorphism $V'_i\cong V'_{i+1}$ for each $i\in I$. It follows that $\dim_I V' = r'\mathbf{1}$ for some $r'\in \N$. Let $x' = x|_{V'}\in \End(V')$ and $x'' = x|_{V''}\in \End(V'')$ and set $\fg' = \fg(V')$ and  $\fg'' = \fg(V'')$. The inclusion $\fg'\times\fg''\subseteq \fg$ induces $\Z/m\Z$-gradings on $\fg'$ and $\fg''$. Then, we have $x'\in \fg'_1$, $x''\in \fg''_1$ and 
	\[
		Z_{\fg_0}(x) = Z_{\fg'_0}(x')\times Z_{\fg''_0}(x'').
	\]
	The $G_0$-distinguished condition on $x$ implies that $Z_{\fg''_0}(x'')$ has no non-zero semisimple element. 
	Since $x'' = x|_{V''} = x_n|_{V''}$ is nilpotent, it lies in $\cO_{\bs_{(A',B')}}$ for  some pair $(A',B')\in \Pf(\N)^2$. By \autoref{lemm:dist-orb-classical}, there exists a unique pair $(A,B)\in \Pf(\N)^2$ such that $F(\bd_{A,B}) = 0$ and $\bd_{A',B'} = \bd_{A,B} + s\mathbf{1}$ with $s = F(\bd_{A', B'})\in \N$. It follows that $\bd = \bd_{A,B} + r\mathbf{1}$, where $r = r' + s$. This proves the first statement. Now, we have $x_s\in \fg'_1$, so $\fZ_{x}\subseteq \fg'$, and $\dim \fZ_{x,1} \le \dim \fg'_1 /\!/G'_0 = r'\le r$. This proves the second statement.\par
	Assume now that $\dim \fZ_{x,1}= r$. Then $\fZ_{x,1}$ is a Cartan subspace of $\fg'_1$ and $r' = r$; therefore, we have $\fh_{x} = \fg''$ and $(A',B') = (A,B)$. \autoref{lemm:dist-orb-classical} implies that $\fg''_*$ satisfies $\dim\fg''_0=\dim\fg''_1$, and $\cO_{\bs_{(A,B)}}$ is the unique $G_0$-distinguished nilpotent orbit in $\fg''_{1}$. 
\end{proof}

\section{Classification of bi-orbital cuspidal character sheaves: proofs of~\autoref{thm:bisup-classical} and~\autoref{thm:bisup-excep}}\label{sec-bicus}
In this section, we give a complete classification of bi-orbital cuspidal character sheaves on graded simple Lie algebras following the general method described in~\S\ref{ssec:method}, proving~\autoref{thm:bisup-classical} and~\autoref{thm:bisup-excep}. We omit the explicit description of the local systems that give rise to bi-orbital cuspidal character sheaves. This can be reduced easily to the description of cuspidal local systems in the ungraded case \cite{L84,LS85}.

\subsection{Classical types: proof of~\autoref{thm:bisup-classical}}
 
Assume that $(G, \sigma)$ is in one of the families: $\bA, \bC, \BD, \iiA$. We discuss the possibilities for the data $(e, \cO, \cC, L, \fl_*, \cO_0, \cC_0)$, making use of the quiver description of graded Lie algebras of classical types (see~\autoref{sec:quiver}). 

By~\autoref{lem:dim}, we can exclude type II graded Lie algebras. Alternatively,  there are no distinguished elements for type II graded Lie algebras  
by~\autoref{prop:distinguished-II/III} (2), so there are no cuspidal character sheaves on $\fg_1$. 

\subsubsection{Type $\bA$}
Let $G = \SL_n$. 
The representation $e\in \fg_1$ of the $\Z/m\Z$-quiver is cyclic and is labelled by a segment $[a,b]$ such that $b-a+1=n$. The graded dimension of the quiver is then given by
\[
\dim V_j = \#\{k\in [a,b]\mid j \equiv k \pmod{m}\}.
\]
It is easy to show that $\dim \fg_1 - \dim \fg_0 = 1$ when $m\mid n$ and $\dim \fg_1 - \dim \fg_0 = 0$ otherwise. In the latter case, the bi-orbital supercuspidal pairs are $(G_0e,\cE_\psi)$, where $G_0e$ is the only distinguished orbit in $\fg_1$, which is regular in $\Lg$, and $\cE_\psi$ are the clean local systems on $G_0e$  parametrised by the primitive characters of the centre $\psi:Z(G) = \mu_n\to \Cc$.

\subsubsection{Type $\mathbf{I}$}For $n\in\N$, recall that we write $\delta_n=n-1-2\lfloor\frac{n-1}{2}\rfloor$. That is $\delta_n=1$ if $2\mid n$, and  $\delta_n=0$ if $2\nmid n$. Recall also $\eta(n)$ from~\eqref{eqn-eta}.

\begin{prop}\label{prop:biorbital-classical}
Let $\cO\subseteq \fg_1^{\nil}$ be a nilpotent orbit. Then, $\cO$ is the supporting stratum of a bi-orbital cuspidal character sheaf on $\fg_1$ if and only if $\cO=\cO_{\bs_{(A,B)}}$ (see~\eqref{multi-I})  
for some $(A, B)\in \Pf(\N)^2$ of the following form:
\begin{gather*}
(A,B) = \begin{cases}(\{0, 1, \ldots, a-1\}, \{0, 1, \ldots, b-1\})&  \BDI^{\varepsilon=1}, \CI, \iiAI \\ (\{2i+\delta_a,\,0\leq i\leq\lfloor\frac{a-1}{2}\rfloor\},\{2i+\delta_b,\,0\leq i\leq\lfloor\frac{b-1}{2}\rfloor\}) & \BDI^{\varepsilon=-1} \end{cases} 
\end{gather*}
for some integers $a, b\ge 0$ satisfying
\[
    \begin{cases}
    a + b \le l & \BDI^{\varepsilon=1}, \iiAI \\
    a + b < l & \CI \\
    a + b \le l \text{ or }(2  \mid(l-a-b) \text{ and } |a - b|\le l) & \BDI^{\varepsilon=-1}. \\
    \end{cases}
\]
Moreover, in this case, every bi-orbital cuspidal character sheaf on $\fg_1$ has supporting stratum $\cO$ and the number of such sheaves is
\[
    \begin{cases}
    1 &  \BDI^{\varepsilon=1}, \CI, \iiAI \\
    \eta(a)\eta(b) & \BDI^{\varepsilon=-1}. \\
    \end{cases}
\]

\end{prop}
\begin{proof}
We follow the method and the notation in~\autoref{ssec:method}. Assume that $\cO$ admits a supercuspidal local system.
Let $e\in \cO$, $h\in \fg_0$, $f\in \fg_{-1}$ and $\varphi:\Cc\to G_0$ be as in~\autoref{ssec:method}. Then, $\varphi$ induces a decomposition of $V$:
\[
V = \bigoplus_{\substack{i\in \frac{1}{2}\Z \\ 0 \le i < m_0}}V[i],\quad V[i] = \bigoplus_{j\in I} \bigoplus_{\substack{k\in \Z\\ \frac{k}{2}\equiv j-i \pmod{m_0}}} \li^\varphi_k{V}_j.
\]
We have $(V[i], V[j]) = 0$ whenever $i + j \not\equiv 0 \pmod{m_0}$. The $\sigma$-pseudo-Levi subalgebra $\fl$ can be described as
\[
\fl = \prod_{0<i<m_0/2}\fgl(V[i])\times \begin{cases}\fso(V[0])\times \fso(V[m_0/2]) & \BDI\\ 
\fsp(V[0])\times \fsp(V[m_0/2]) &\CI \\
\fso(V[0], \omega_{\varphi})\times\fsp(V[m_0/2], \omega_{\varphi}) &\iiA
\end{cases}, 
\]
where $\omega_{\varphi}:V\times V\to \C$ is a new bilinear form defined by $\omega_{\varphi}(v, w) = (v, \varphi(\zeta_{2m})w)$.

In particular, the condition that $L$ is semisimple implies that $V[i] = 0$ whenever $i\notin \{0,m_0/2\}$. Set $V' = V[0]$ and $V'' = V[m_0/2]$. Then, we have $V = V'\oplus V''$ and
\[
L = \begin{cases}(\Spin(V')\times\Spin(V''))/(\varepsilon_{V'},\varepsilon_{V''}) & \BDI\\ 
\Sp(V')\times\Sp(V'') &\CI \\
\SO(V', \omega_{\varphi})\times\Sp(V'', \omega_{\varphi}) &\iiA
\end{cases}, 
\] 
Let $\cC$ be a supercuspidal local system on $\cO$. Then, $(\cO, \cC)$ induces a cuspidal pair $(\cO_L, \cC_L) = (\cO\cap \fl, \cC|_{\cO\cap \fl})$ on $\fl = \Lie L$. We can write $\cO_L = \cO'\times \cO''$ and $\cC_L = \cC'\boxtimes \cC''$, where $(\cO', \cC')$ (resp.  $(\cO'', \cC'')$) is a cuspidal pair on the $V'$-factor (resp. $V''$-factor) of $\fl$. Note that in the case $\BDI$, the element $(\varepsilon_{V'}, 1) = (1, \varepsilon_{V''})\in Z(L)$ is sent to $\varepsilon_{V} \in Z(G)^{\theta}$ under the inclusion map $L\to G$; therefore, $\varepsilon'$ and $\varepsilon''$ act by the same sign on $\cC'$ and $\cC''$. 
According to~\cite[\S 12.4]{L84}, the element $e\in \cO_L$ operates on $V'$ (resp. $V''$) with Jordan blocks of sizes given by a partition $\lambda'$ (resp. $\lambda''$) of the form 
\begin{gather*}
\lambda' = \begin{cases}(1, 3, \ldots, 2a-1) & \BDI^{\varepsilon=1}, \iiAI\\ 
(2, 4, \ldots, 2a)  &\CI \\
(1, 5,\ldots, 2a-1)  & \BDI^{\varepsilon=-1} \text{ and } 2\nmid a \\
(3, 7,\ldots, 2a-1) & \BDI^{\varepsilon=-1} \text{ and } 2\mid a \\
\end{cases} \\
\lambda'' = \begin{cases}(1, 3, \ldots, 2b-1) & \BDI^{\varepsilon=1}\\ 
(2, 4, \ldots, 2b)  &\CI, \iiAI \\
(1, 5,\ldots, 2b-1) & \BDI^{\varepsilon=-1} \text{ and } 2\nmid b \\
 (3, 7,\ldots, 2b-1) & \BDI^{\varepsilon=-1} \text{ and } 2\mid b  \\
\end{cases},
\end{gather*}
for some integer $a\ge 0$ (resp. $b\ge 0$).  In the cases $\BDI^{\varepsilon=1}, \CI, \iiAI$, there is a unique cuspidal local system on the orbit $\cO_L$. In the case $\BDI^{\varepsilon=-1}$, the number of cuspidal local systems on the orbit $\cO_L$ is $\eta(a)\eta(b)$, where $\eta(n)$ is as in~\eqref{eqn-eta}. Consider the maps
\[
\{\text{clean loc. sys. on $\cO_L$}\}\xrightarrow{\text{res}}\{\text{clean loc. sys. on $\cO_L\cap \fl_1 = \cO\cap \fl_1$}\}\xleftarrow{\text{res}}\{\text{clean loc. sys. on $\cO$}\}.
\]
The first restriction map is a bijection by~\cite[\S 4]{L95}, and the second is a bijection by~\cite[\S 2.9(c)]{LY17a}.  It follows that $\cO=\cO_{\bs_{(A,B)}}$ for some $(A,B)$ as in the proposition and the number of supercuspidal local systems on $\cO$ is as stated. It remains to determine when the condition $\dim \fg_1 = \dim \fg_0$ holds. 

Recall the function $F$ from~\eqref{eq:F}, which satisfies $F(\bd_{A,B}) = \dim \fg_1 - \dim \fg_0$. The condition on $(a,b)$ follows from~\autoref{lemm:dist-orb-classical} (2) and   the condition that $\dim \fg_1 = \dim \fg_0$.  
\end{proof}

\subsubsection{Types $\mathbf{III}$}
Since every distinguished element is nilpotent by \autoref{prop:distinguished-II/III} (3), we deduce that every cuspidal character sheaf is bi-orbital supercuspidal. The $\sigma$-pseudo-Levi subgroup $L$ defined in~\autoref{ssec:method} is of the form
\[
L = \begin{cases}\Spin(V) & \BDIII \\ \Sp(V) & \CIII \\ \SO(V) & \iiAIIIi \\ \Sp(V) & \iiAIIIii \end{cases}
\]
The element $e\in \cO_L$ operates on $V$ with Jordan blocks of sizes given by a partition $\lambda$ of the form 
\[
\lambda = \begin{cases}(1, 3, \ldots, 2r-1) & \BDIII^{\varepsilon=1}, \iiAIIIi\\ 
(2, 4, \ldots, 2r)  &\CI, \iiAIIIii\\
(1, 5, \ldots, 2r - 1) \text{ or } (3, 7, \ldots, 2r-1) & \BDIII^{\varepsilon=-1}
\end{cases}
\]
for some $r\ge 0$. In the cases $\BDIII^{\varepsilon=1}, \CIII, \iiAIII$, there is a unique cuspidal local system on the orbit $\cO_L$. In the cases $\BDIII^{\varepsilon=-1}$, the number of cuspidal local systems on the orbit $\cO_L$ is $1 + \eta(r)$.

\subsection{Exceptional types: proof of~\autoref{thm:bisup-excep} }

Recall notations from~\autoref{ssec:affine} and~\autoref{ssec:nilp}.
 We aim to classify the following set 
\[
\Pi := \{(x, \cO, \cC) \mid x\in \kappa,\; (\cO, \cC)\text{ a supercuspidal pair on $\fg_{x,1}$} \}.
\]

Given any affine simple root $\beta\in \Delta$, let $L^{\beta}$ denote the semisimple subgroup of $G$  generated by $T_0$ and the root subgroups attached to $\pm\Delta \setminus\{\pm\beta\}$; it is called a maximal $\sigma$-psuedo Levi subgroup of $G$. It comes with a Borel subgroup $B^{\beta}$ generated by $T_0$ and the root subgroups attached to $\Delta \setminus\{\beta\}$. Moreover, every point $x\in V$ induces a $\Z$-grading $\fl^{\beta}_{x, *}$ on the Lie algebra $\fl^{\beta}$ of $L^\beta$ such that $\fl_{x, i}\subseteq \fg_{x, \underline{i}}$ for $i\in \Z$, where $\underline{i} = i + m_x\Z\in \Z/m_x\Z$. 

Given any distinguished nilpotent $L^{\beta}$-orbit $\cO_0\subseteq \fl^{\beta}$,  
recall from~\autoref{lem:dist-orbit} the weight function $\rho: \Delta_{L^\beta}=\Delta\setminus\{\beta\}\to \{0,1\}$ associated with $\cO_0$. 
Assume that $\cO_0\cap \fl^{\beta}_{x,1} \neq\emptyset$. Then, the intersection $\cO := \cO_0\cap \fl^{\beta}_{x,1}$ is open and dense in $\fl^{\beta}_{x,1}$ and the grading on $\fl^{\beta}_{x, *}$ is determined by $\rho$: for any root $\gamma$ of $\fl^{\beta}$, the root subspace $\fl^{\beta}_{\gamma}\subseteq \fl^{\beta}_{x,\rho(\gamma)}$,
where for any root $\gamma = \sum_{\alpha\in \Delta_{L^\beta}}n_{\alpha}\alpha$,  $\rho(\gamma):=\sum_{\alpha\in \Delta_{L^\beta}}n_{\alpha} \rho(\alpha).$

\begin{lemm}\label{lemm:pi'}
There is a natural bijection from $\Pi$ to the following set:
\[
\Pi' = \{(\beta, m, \cO_0, \cC_0)\;;\; \beta\in \Delta,\; m\in \Z_{>0},\; (\ord \sigma)| m,\; (\cO_0, \cC_0)\text{ cuspidal pair on $\fl^{\beta}$},\; \text{\eqref{eq:div}}
\}
\]
where
\begin{equation}\label{eq:div}
b_\beta \mid \left(\frac{m}{\ord \sigma} - \sum_{\alpha\in \Delta\setminus \{\beta\}} b_{\alpha}\rho(\alpha)\right)\,\tag{*}
\end{equation}
(see~\eqref{eqn-bi} for the definiton of $b_\alpha$'s) and $\rho:\Delta_{L^\beta}\to\{0,1\}$ is the weight function determined by $\cO_0$ as above.
 
\end{lemm}
\begin{proof}
The affine roots $\Phi$ define a hyperplane arrangement on $V\otimes_\Q\R$. Let $\Sigma\subseteq V$ be the set of vertices. 
We introduce an intermediate set:
\[
\Pi_1 := \{(x, \cO, \cC)\;;\; x\in V, \; (\cO, \cC)\text{ supercuspidal pair on $\fg_{x,1}$} \} 
\]
and define a $\Waff$-action as follows. Since the underlying finite Weyl group of $\Waff$ can be identified with the Weyl group $W(G^{\sigma}, T_0)$, for each $w\in \Waff$, we can choose a lifting $\dot w\in N_{G^\sigma}(T_0)$ of the image of $w$ in $W(G^{\sigma}, T_0)$. Then, $\Waff$ acts on $\Pi_1$ by $w(x, \cO, \cC) = (wx, \dot w\cO, \dot w_*\cC)$. The map
\[
\Pi\to \Pi_1,\quad  (x, \cO, \cC) \mapsto (x, \cO, \cC)
\]  
induces a bijection $\Pi\xrightarrow{\sim}\Pi_1/\Waff$. \par

We now construct a bijection $\Pi_1/\Waff\xrightarrow{\sim} \Pi'$. Given $(x, \cO, \cC)\in \Pi_1$, the point $x\in V$ determines a $\Z/ m_x\Z$-grading $\fg_{x, *}$ as explained in~\autoref{ssec:affine}. Choose an $\fsl_2$-triple $(e, h, f)$ with $e\in \cO$, $h\in \ft_0$ and $f\in \fg_{x,-1}$. It determines a cocharacter $\iota:\Cc\to T_{0}$ such that $d\iota(1) = h$. The point 
\[
v:= x -\tfrac{1}{2m_x}\iota\in V
\]
lies in $\Sigma$. Choose $w\in \Waff$ such that $v\in w\kappa$. Then, there is a unique simple affine root $\beta\in \Delta$ characterised by $\beta(w^{-1}v) > 0$. Note that $\beta$ is independent of the choice of $w$. Up to replacing $x$ with $w^{-1}x$, we may assume $v\in \kappa$. \par

Let $\fl=\on{Lie}(L^\beta)$. The point $x\in V$ yields a $\Z$-grading $\fl_{x, *}$. The restriction $(\cO\cap \fl_{x, 1}, \cC|_{\cO\cap \fl_{x, 1}})$ is a cuspidal pair on $\fl_{x, 1}$, so there exists a unique  cuspidal pair $(\cO', \cC')$ on $\fl_x$ such that $\cO'\cap \fl_{x, 1} = \cO\cap \fl_{x, 1}$ and $\cC'|_{\cO'\cap \fl_{x, 1}} = \cC|_{\cO\cap \fl_{x, 1}}$. Set $(\cO_0, \cC_0) = (\dot w \cO', \dot w_* \cC')$. Then, the map $\Pi_1\to \Pi'$ sending $(x, \cO, \cC)$ to $(\beta, m_x, \cO_0, \cC_0)$ factors through the quotient $\Pi_1\to \Pi_1/\Waff$. This yields a map $a:\Pi_1/\Waff\to \Pi'$. It remains to show that $a$ is a bijection. \par

To construct the inverse map, let $(\beta, m, \cO_0, \cC_0)\in \Pi'$. Let $v\in \kappa$ be the unique vertex such that $\alpha(v) = 0$ for $\alpha\in \Delta\setminus\{\beta\}$. As in~\autoref{ssec:nilp}, we may find an $\fsl_2$-triple $(e, h, f)$ in $\fl^{\beta}$ with $e\in \cO_0$ such that the image of the associated cocharacter $\iota:\Cc\to L^{\beta}$ lies in $T_0$ and $\langle \alpha, \iota\rangle \ge 0$ for $\alpha\in \Delta\setminus\{\beta\}$. The weight function given by
\[
\rho:\Delta\setminus\{\beta\}\to \{0,1\},\quad \rho(\alpha) = \langle \alpha, \iota\rangle/2.
\]
and there exists $\alpha$ such that $\rho(\alpha) = 1$.  
Set 
\[
x = v + \tfrac{1}{2m}\iota, 
\]
which gives rise to a $\Z$-grading $\fl^{\beta}_{x, *}$ satisfying $\fl^{\beta}_{\alpha}\subseteq \fl^{\beta}_{x,\rho(\alpha)}$ for $\alpha\in \Delta\setminus\{\beta\}$. We have $m = m_x$: indeed, $m\alpha(x) = \rho(\alpha)\in \{0,1\}$ for $\alpha\in \Delta\setminus\{\beta\}$ and $m\alpha(x)=1$ for some $\alpha$, and, in view of~\autoref{eqn-bi}, the condition~\eqref{eq:div} implies that $m\beta(x) \in \Z$. It follows that $\fl^{\beta}_{x, 1}\subseteq \fg_{x, 1}$. The restriction $(\cO_0\cap \fl^{\beta}_{x, 1}, \cC|_{\cO_0\cap \fl^{\beta}_{x, 1}})$ is a cuspidal pair on $\fl^{\beta}_{x, 1}$ and thus extends to a supercuspidal pair $(\cO, \cC)$ on $\fg_{x, 1}$. Then, we have a map $\Pi'\to \Pi_1$ sending $(\beta, m, \cO_0, \cC_0)$ to $(x, \cO, \cC)$ and it induces a map $b:\Pi'\to \Pi_1/\Waff$. It is not hard to check that $a$ and $b$ are inverse to each other.

\end{proof}

\begin{lemm}\label{lem:cover}
Let $\beta\in \Delta$ and let $p:\tilde L\to L^\beta$ be a universal cover. Then, 
\[
    \ker p = \left\{\prod_{\alpha\in \Delta\setminus\{\beta\}}\alpha^{\vee}(\zeta^{c_{\alpha}})\;;\;\zeta^{c_{\beta} }= 1\right\}\subseteq Z(\tilde L).
\]
where $(c_{\alpha})_{\alpha\in \Delta}$ are the positive integers determined by the following conditions:
\[
    \sum_{\alpha\in \Delta} c_{\alpha}\alpha^{\vee}= 0,\quad \gcd(c_{\alpha})_{\alpha\in \Delta} = 1.
\]
\end{lemm}
\begin{proof}
Let $\tilde T_0$ be the inverse image of $T_0\subseteq L^{\beta}$ in $\tilde L$. Then, the natural map $\bX_*(\tilde T_0)\to \bX_*(T_0)$ is an inclusion of sublattice of finite index. Since $G$ is simply connected, $\{\alpha^{\vee}\}_{\alpha\in\Delta}$ generates $\bX_*(T_0) = \bX_*(T)^{\sigma}$ as abelian group (see~\cite[\S 8.3]{Ka} for the cases $\sigma\neq 1$.) It follows that 
\[
    \Z / c_{\beta}\Z\xrightarrow{\sim}\bX_*(T_0) / \bX_*(\tilde T_0),\quad k + c_{\beta}\Z\mapsto k\sum_{\alpha\in \Delta\setminus\{\beta\}}c_{\alpha}\alpha^{\vee}
\]
and thus the kernel of $\tilde L\to L^{\beta}$ is given by $\{\prod_{\alpha\in \Delta\setminus\{\beta\}}\alpha^{\vee}(\zeta^{c_{\alpha}})\;;\;\zeta\in \mu_{c_{\beta}}\}$. 
\end{proof}

\begin{lemm}\label{lem:rank0}
Suppose that $(\beta, m, \cO_0, \cC_0)\in \Pi'$ satisfies $m/(\ord \sigma) > h$, where $h = \sum_{\alpha\in \Delta}b_{\alpha}$ is the (twisted) Coxeter number for $(G, \sigma)$. Then, its image $(x, \cO, \cC)\in \Pi$ satisfies $\fg_{x, 0} = \fl^{\beta}_{x, 0}$ and $\fg_{x, 1} = \fl^{\beta}_{x, 1}$. In particular, $\dim \fg_{x, 0}  = \dim\fg_{x, 1}$. 
\end{lemm}
\begin{proof}
    When $m/\ord \sigma > h$, we have 
    \[
        \frac{m}{\ord \sigma} - \sum_{\alpha\in \Delta\setminus \{\beta\}} b_{\alpha}\rho(\alpha) \ge \frac{m}{\ord \sigma} - h + b_{\beta} > b_{\beta}.
    \]
    Therefore, the Kac coordinates $(n_{\alpha})_{\alpha\in \Delta}$ given by $n_{\alpha} = \rho(\alpha)$ for $\alpha\in \Delta\setminus\{\beta\}$ and $n_{\beta} = b_{\beta}^{-1}(\frac{m}{\ord \sigma} - \sum_{\alpha\in \Delta\setminus \{\beta\}} b_{\alpha}\rho(\alpha))$ are normalised so that $n_\alpha\in \N$ for each $\alpha\in \Delta$ and $\gcd\{n_{\alpha}\mid \alpha\in \Delta\} = 1$. 
    Moreover, we have $n_{\beta} > 1$. Let $\Phi^{\beta} = \Z( \Delta\setminus\{\beta\})\cap \Phi$. The Kac coordinates $(n_{\alpha})_{\alpha\in \Delta}$ yield a grading $\Phi = \bigsqcup_{n\in \Z} \Phi(n)$ satisfying $\alpha\in \Phi(n_{\alpha})$ for $\alpha\in \Delta$. Setting $\Phi^{\beta}(n) = \Phi^{\beta}\cap \Phi(n)$, we have $\Phi^{\beta}(0) = \Phi(0)$ and $\Phi^{\beta}(1) = \Phi(1)$, so that
    \[
    \fg_{x, 0} = \ft_0\oplus \bigoplus_{\alpha\in \Phi^{\beta}(0)}\fg_{\alpha} = \fl^{\beta}_{x,0}, \quad
    \fg_{x, 1}  = \bigoplus_{\alpha\in \Phi^{\beta}(1)}\fg_{\alpha} = \fl^{\beta}_{x,1}. 
    \]
    Since $L^{\beta}$ is semisimple and the $\Z$-grading on $\fl^{\beta}_{x, *}$ arises from a distinguished nilpotent element, we have $\dim\fl^{\beta}_{x, 0} = \dim\fl^{\beta}_{x, 1}$ (see~\cite[5.7.5]{C93}, for example). 
\end{proof}
Now, we describe our method to classify bi-orbital cuspidal sheaves on graded Lie algebras of exceptional types. 
For each pair $(G, \sigma)$, where $G$ is a simply connected almost simple group endowed with a pinning and $\sigma$ is a pinned automorphism, we can enumerate the set $\Pi'$ and check for each element $(\beta, m, \cO_0, \cC_0)\in \Pi'$ whether its image $(x, \cO, \cC)\in \Pi$ satisfies $\dim \fg_{x, 0} = \dim \fg_{x, 1}$.  \par

There are finitely many possible $(\beta, \cO_0, \cC_0)$ which can appear. To enumerate the $(\beta, \cO_0, \cC_0)$, fix $\beta\in \Delta$ and let $\tilde L\to L^\beta$ be a universal cover. The cuspidal pairs $(\cO_0, \cC_0)$ on $\tilde L$ have been classified by Lusztig~\cite{L84}. In order to decide whether $\cC_0$ descends to an $L^{\beta}$-equivariant local system, we use \autoref{lem:cover} to determine whether the kernel of the universal cover $p:\tilde L\to L^{\beta}$ acts trivially on $\cC_0$. \par
On the other hand, given $(\beta, \cO_0, \cC_0)$, \autoref{lem:rank0} implies that we only need to check the identity $\dim\fg_{x,0} = \dim\fg_{x,1}$ for finitely many $m\in \Z_{\ge 1}$ (up to the twisted Coxeter number.) \par

We obtain the set $\Pi'$ tabulated as in~\autoref{thm:bisup-excep}, where the column Kac shows the element $x\in \kappa$ in its normalised Kac 
coordinates and $L=L^\beta$.

\section{Classical types: proof of~\autoref{thm-classical}}\label{sec:classical}
Using the strategy described in~\autoref{ssec:strategy}, we see that~\autoref{thm-classical} follows from~\autoref{prop:distinguished-II/III},~\autoref{prop:distinguished-I}, and~\autoref{thm:bisup-classical}. To give more details, in this section, we demonstrate how to work this out in the case of cuspidal character sheaves with non-trivial $\varepsilon$-action in type $\BDI$. 

Assume now $(G, \theta)$ is of type $\BDI$, so that $I = \Z/m\Z$ with $m = 2l$. We may write $G = \Spin(V)$, where $V$ is equipped with a quadratic form and a grading $V = \bigoplus_{i\in I} V_i$. 
\begin{prop}
Assume that $\cF$ is a cuspidal character sheaf on $\fg_1$ such that $\varepsilon$ acts as $-1$ on $\cF$. Let $\cS$ be the supporting stratum of $\cF$. Then, given any $x = x_s + x_n\in \cS$, we have $V = V' \oplus V''$ with $V' = x_s(V)$ and $V'' = \ker(x_s)$ satisfying
\begin{enumerate}
    \item $\dim_I V' = r\delta$,
    \item $x_n|_{V'} = 0$ and $x_s|_{V'}$ is regular semisimple, and 
    \item the element $x_n|_{V''}\in \fg(V'')$ lies in the orbit labelled by the multi-segment $\mathbf{s}-\sum_{k=0}^{2l-1}r[k,k]$, where $\mathbf{s}$ is as in~\autoref{thm-classical} (3) (b). 
\end{enumerate}
Conversely, if $x\in \fg_1$ is of the form described above, then the stratum $\cS\subseteq \fg_1$ containing $x$ is the supporting stratum of some cuspidal character sheaf on which $\varepsilon$ acts as $-1$.
\end{prop}

\begin{proof}
Let $\cS\subseteq \fg_1$ be the stratum of a cuspidal character sheaf $\cF$ and let $x = x_s + x_n\in \cS$. Then we have an orthogonal decomposition $V = V'\oplus V''$, where $V' = \im(x_s)$ and $V'' = \ker(x_s)$. Write $\cF = \IC(\cF^{\circ})$, where $\cF^{\circ} = \cF|_{\cS}[-\dim \cS]$. 
Set $G'' = \Spin(V'')$ and $x'' = x|_{V''}$. Then, $\cF^{\circ} |_{G''x_n''}$ is a sum of bi-orbital supercuspidal local systems with $\varepsilon = -1$, which implies that $(V'', x_n|_{V''})$ is of the form described in~\autoref{prop:biorbital-classical}. Moreover, we have $\dim \fg''_1 = \dim \fg''_0$ due to the existence of bi-orbital supercuspidal sheaves, which implies $\dim \fg_1 - \dim \fg_0 = \dim V'_0$. The condition $\dim\fZ_{x, 1} = \dim \fg_1 - \dim \fg_0 = \dim V'_0$ implies $x_s|_{V'}$ is regular semisimple, thus $x_n|_{V'} = 0$. \par
Conversely, let $x\in \fg_1$ be of the described form. Then, setting $H = Z_{G}(x_s)^{\mathrm{der}}$, $\fh_1$ admits a bi-orbital $H_0$-equivariant sheaf, and $\varepsilon$ acts on every such sheaf by $-1$. The existence of cuspidal character sheaves with supporting stratum $\cS$ and $\varepsilon = -1$ follows from the nearby-cycle construction.
\end{proof}

\section{Exceptional types: proof of~\autoref{thm-excp} and~\autoref{coro:weyl}}\label{sec:exceptional}
In this section we prove~\autoref{thm-excp} and~\autoref{coro:weyl} for exceptional types. We make use of the construction of automorphisms $\theta:G\to G$ in~\cite{RLYG}. Recall that $r=\dim\Lg_1-\dim\Lg_0$.

In view of~\autoref{table 2},~\autoref{thm-excp} for the case when $r=0$ follows from~\autoref{thm:bisup-excep} and there is nothing to prove for~\autoref{theo:cuspidal-cartan} as the cuspidal $\fc$-stratum  $\fd=0$. It remains to consider the cases when $r=\dim\fc>0$, the grading $2_a$ of $E_6$ and the grading $3_a$ of $E_7$. When $r=\dim\fc>0$,~\autoref{theo:cuspidal-cartan} is immediate as the cuspidal $\fc$-stratum  $\fd=\fc$, the Cartan subspace.
\subsection{The case when \texorpdfstring{$r=\dim\fc>0$}{r = dim c > 0}}
Suppose that $r=\dim\fc>0$. If the grading is GIT stable, then as in~\cite{LTVX}, all cupsidal character sheaves have full support. If the grading is non GIT stable, then comparing~\autoref{table 1} and~\autoref{table 2}, we see that, to prove~\autoref{thm-excp}, it suffices to show that $\theta|_{Z_\Lg(\fc)_{\on{der}}}$ affords a bi-orbital cuspidal character sheaf with support equal to $(Z_\Lg(\fc)_{\on{der}})_1$ for all gradings in Table~\ref{table 2} except for the grading $10_a$ of $E_7$.  

These claims follow from the following proposition.

\begin{prop}\label{prop-centralisers}
For the gradings in the third column of Table~\ref{table 2}, the type of the root system of $Z_\Lg(\fc)_{\on{der}}$ and $\theta|_{Z_\Lg(\fc)_{\on{der}}}$ are given as follows. 
\begin{longtable}{p{1cm}|p{2cm}|p{3.2cm}|p{7.8cm}}
\hline
Type&grading&type of $Z_\Lg(\fc)_{\on{der}}$&$\theta|_{Z_\Lg(\fc)_{\on{der}}}$\\
\hline
$^2E_6$&$10_b$&$A_1$&$\theta_{10}$\\
\hline
$E_7$&$4_a$&$A_1\times A_1\times A_1$&restricts to each factor $A_1$ as $\theta_4$
 \\\hline
$E_7$&$8_a$&$A_1\times A_1\times A_1$& permutes the first two factors of $A_1$ with $\theta^2|_{A_1}=\theta_4$ and $\theta$ restricts to the third factor of $A_1$ as $\theta_8$\\\hline
$E_7$&$10_a$&$A_2$&outer of order $10$ with quiver $\xymatrix@R=.5pc{&1\ar[r]&0\ar[dd]\\1\ar[ur]&&\\&1\ar[ul]&0\ar[l]}$\\\hline
$E_7$&$10_b$&$A_2$&outer of order $10$ with quiver $\xymatrix@R=.5pc{&0\ar[r]&1\ar[dd]\\1\ar[ur]&&\\&0\ar[ul]&1\ar[l]}$\\\hline
$E_7$&$12_a$&$A_1\times A_1\times A_1$&permutes three factors of $A_1$ and $\theta^3|_{A_1}=\theta_4$\\
\hline
$E_8$&$9_a$&$A_2$&$\theta_9$\\
\hline
$E_8$&$12_e$&$D_4$&$^{3}D_4$ with Kac diagram $\xymatrix{{\substack{1\\\circ\\\,}}&{\substack{0\\\circ\\\,}}\ar@{-}[r]\ar@3{<-}[l]&{\substack{3\\\circ\\\,}}}$\\
\hline
$E_8$&$14_a$&$A_1$&$\theta_{14}$\\
\hline
$E_8$&$18_c$&$A_2$&outer of order $18$ with quiver \xymatrix@R=3mm{&0\ar[r]&0\ar[r]&0\ar[r]&1\ar[dd]\\1\ar[ur]&&&&\\&0\ar[ul]&0\ar[l]&0\ar[l]&1\ar[l]}\\
\hline

\end{longtable}
In the above we have written $\theta_m=\on{int}\left(\begin{matrix}\zeta_{2m}&\\&\zeta_{2m}^{-1}\end{matrix}\right)$ for the inner automorphism of order $m$ in type $A_1$, and $\theta_m=\on{int}\left(\begin{matrix}\zeta_{m}&&\\&1&\\&&\zeta_{m}^{-1}\end{matrix}\right)$ for the inner automorphism of order $m$ in type $A_2$. Note that $\theta_m$ affords bi-orbital cuspidal character sheaf if and only if $m\geq 3$ in type $A_1$, and $m\geq 4$ in type $A_2$.
\end{prop}

We fix a canonical Cartan subalgebra $\Lt\supset\La$ of $\Lg$ as in~\cite[\S3.1]{RLYG} and let $T=Z_G(\Lt)$ the corresponding maximal torus of $G$. Let $W_G=N_G(T)/T$ be the Weyl group of $G$. Let $\Phi\subset\Lt^*$ be the set of roots and $\Delta$ a set of simple roots. We write $\Lg_\alpha$ for the root space corresponding to $\alpha\in\Phi$. We choose a set of Chevalley basis of $\Lg$ consisting of root vectors $X_\alpha\in\Lg_\alpha$, $\alpha\in\Phi$, and $h_{\alpha_i}\in\Lt$, $\alpha_i\in\Delta$ (see for example~\cite[\S25]{H}). In particular, we have
\bern
&&[X_{\alpha_i},X_{-\alpha_i}]=h_{\alpha_i},\ [X_\alpha,X_{-\alpha}]:=h_\alpha\in{\bZ}\text{-span of }h_{\alpha_i},\\
&& [X_\alpha,X_\beta]=n_{\alpha,\beta}X_{\alpha+\beta}\text{ for some integer $n_{\alpha,\beta}$ if $\alpha+\beta\in\Phi$}.
\eern

Suppose that $G$ is of type $E_7$ or $E_8$. Then each positive grading of $\Lg$ can be induced from $\theta=\on{Int}n_w:G\to G$ for some $n_w\in N_G(T)$ and the choices of $w$ are given in~\cite[Table 30-31]{RLYG}. 

Let now $\theta=\on{Int}n_w:G\to G$. We will also write $\theta$ for $d\theta:\Lg\to\Lg$. Let $\Lg_i\subset \Lg$ be the eigenspace of $\theta$ with eigenvalue $\zeta_m^i$, where $\zeta_m$ is a fixed primitive $m$-th root of 1. Since $T$ is $\theta$-stable, $\theta$ induces a bijection $\Phi\to\Phi,\alpha\mapsto\theta\alpha$ such that $\theta(\Lg_\alpha)=\Lg_{\theta\alpha}$. In fact, $\theta\alpha=w\alpha$. Let us write
\beqn
\Lg_{\{\alpha\}}=\sum_{i\in\mathbb{Z}_{\geq 0}}\Lg_{\theta^i\alpha}\text{ and }d_\alpha=\dim\Lg_{\{\alpha\}}.
\eeqn
It is clear that $d_\alpha$ divides the order of $\theta$. We write $c_\alpha\in\bC^*$ for the constants such that
\beq
\theta(X_\alpha)=c_\alpha X_{\theta\alpha},\,\alpha\in\Phi.
\eeq
Then
\beqn
\left(\prod_{i=0}^{d_\alpha-1}c_{\theta^i\alpha}\right)^{\frac{m}{d_\alpha}}=1,\ c_\alpha c_{-\alpha}=1.
\eeqn

\begin{lemm}\label{lem-root spaces}
Suppose that $\theta=\on{Int}n_w:G\to G$ has order $m$. Then
\begin{enumerate}
\item $Z_\Lg(\fc)\supset\Lt\oplus\sum_{d_\alpha<m}\Lg_{\{\alpha\}}$.  
\item if $d_\alpha=m$, then $\on{dim}\Lg_{\{\alpha\}}\cap\Lg_i=1$ for each $i\in\bZ/m$.
\item if $d_\alpha=d<m$, then $\on{dim}\Lg_{\{\alpha\}}\cap\Lg_i=1$ for each $i\in\bZ/m$ such that $\zeta_m^{id}=\prod_{i=0}^{d-1}c_{\theta^i\alpha}$, and $\on{dim}\Lg_{\{\alpha\}}\cap\Lg_i=0$ otherwise.
\end{enumerate}
\end{lemm}

We will now proceed as follows to prove Proposition~\ref{prop-centralisers}. In each case we choose a $w\in W_G$ as in~\cite[Table 30-31]{RLYG}. 

Step 1. Since we have $\theta(h_{\alpha_i})=h_{w\alpha_i}=\sum n_jh_{\alpha_j}$, where $w\alpha_i=\sum n_j\alpha_j$ (note that we are in the simply-laced situation), it is easy to compute $\Lt_i$, the $\zeta_m^i$-eigenspace of $\Lt$ under $w$, in particular the Cartan subspace $\fc=\Lt_1$. 

Step 2. Compute $Z_\Lg(\fc)$ by evaluating roots on $\fc$.

Step 3. To determine $\theta|_{(Z_\Lg(\fc))_{\on{der}}}$, we compute the orbits of roots found in Step 2 under $w$. We can also easily compute $\dim\Lg_i$ from the Kac diagram of each grading. Making use of Lemma~\ref{lem-root spaces} and comparing with $\dim\Lg_i$'s, we can determine the structure constants $c_\alpha$. We are then able to describe $\theta|_{(Z_\Lg(\fc))_{\on{der}}}$.

We illustrate this procedure in the example of gradings $10_a$ and $10_b$ of $E_7$, and the grading $12_e$ of $E_8$. This also enables us to deal with the grading $10_b$ of $^2E_6$, which can be realised as $10_b$ of $E_7$ restrict to a subalgebra of type $E_6$.

\begin{exam}\label{example-e7}\normalfont The gradings $10_a$ and $10_b$ of $E_7$.

We label the simple roots as follows
\beqn
\xymatrix{{\substack{\alpha_1\\\circ}}\ar@{-}[r]&{\substack{\alpha_3\\\circ}}\ar@{-}[r]&{\substack{\alpha_4\\\circ}}\ar@{-}[r]\ar@{-}[d]&{\substack{\alpha_5\\\circ}}\ar@{-}[r]&{\substack{\alpha_6\\\circ}}\ar@{-}[r]&{\substack{\alpha_7\\\circ}}\\&&{\substack{\circ\\\alpha_2}}&&}
\eeqn
We will write a positive root $\sum_in_i\alpha_i$ as $n_1n_2\cdots n_7$.

Both gradings can be constructed as a lift of $w\in W_G$ (of type $D_6$ in Carter's notation~\cite{Ca}). Let $w=s_{\alpha_1}s_{\alpha_4}s_{\alpha_0}s_{\alpha_3}s_{\alpha_2}s_{\alpha_5}\in W_G$, where $\alpha_0$ is the highest root. Then we have
 \bern
 w:&&\alpha_1\mapsto -1223321,\,\alpha_2\mapsto-0101000,\,\alpha_3\mapsto-1011000,\,\\&&\alpha_4\mapsto1112100,\,\alpha_5\mapsto-0001100,\,\alpha_6\mapsto0001110,\,\alpha_7\mapsto\alpha_7.
 \eern
 We have $\dim\Lt_i=1$, when $i=0,1,3,7,9$, $\dim\Lt_5=2$ and $\dim\Lt_i=0$ otherwise (this can also be deduced from~\cite{Ca} as $w$ is an element of type $D_6$) and
 \beqn
 \fc=\Lt_1=\on{span}\{(1-\zeta_{10},2-\zeta_{10}-\zeta_{10}^{-1},2-\zeta_{10}-\zeta_{10}^{-1},\zeta_{10}^3-2\zeta_{10}^2+2,3-\zeta_{10}-\zeta_{10}^{-1},2,1)\}.
 \eeqn
 Here and below the vectors are in the basis $\{h_{\alpha_i}\}$. 
 
 We check that
 \beqn
 Z_\fg(\fc)=\Lt\oplus\Lg_{\pm1122110}\oplus \Lg_{\pm1122111}\oplus \Lg_{\pm0000001}.
 \eeqn
 where $w(1122110)=-1122111$, $w(1122111)=-1122110$ and $w(0000001)=0000001$.
 Thus $\Lh:=(Z_\fg(\fc))_{\on{der}}$ is of type $A_2$. Note that $h_\beta\in\Lh_0$ and $h_\alpha+h_\beta\in\Lh_5$.
 
 This implies that there are $12$ orbits of size $10$ on the roots. Let us write $E=\Lg_{\pm1122110}\oplus \Lg_{\pm1122111}\oplus \Lg_{\pm0000001}$, $\alpha=1122110$ and $\beta=0000001$.
 
Let us write  $\theta_b:G\to G$ for the automorphism that gives rise to the grading $10_b$. In this case we can check that $\dim\Lg_i=12$, $i=4,6$, $\dim\Lg_i=13$, $i=0,2,8$, and $\dim\Lg_i=14$, $i=1,3,5,7,9$. This implies that $\on{dim}\Lg_i\cap E=1$ when $i=1,2,3,7,8,9$, and $\on{dim}\Lg_i\cap E=0$ otherwise.  Thus $\dim\Lh_i=1$, when $i=0,1,2,3,5,7,8,9$, and $\dim\Lh_6=0$.

We have $(c_\alpha c_{\theta_b\alpha})^5=1$ and $c_\beta^{10}=1$. Applying Lemma~\ref{lem-root spaces}, we see that  $c_\alpha c_{\theta_b\alpha}=\zeta_{10}^4$ or $\zeta_{10}^6$, $c_\beta=\zeta_{10}$ or $\zeta_{10}^{-1}$. Now since $\alpha+\beta=-\theta\alpha$, applying $\theta$ to $[X_\alpha,X_\beta]=n_{\alpha,\beta}X_{-\theta\alpha}$, we see that $c_\alpha c_{\theta_b\alpha}c_\beta\in\mathbb{Q}$. Thus we conclude that either $c_\alpha c_{\theta_b\alpha}=\zeta_{10}^4$ and $c_\beta=\zeta_{10}$, or $c_\alpha c_{\theta_b\alpha}=\zeta_{10}^6$ and $c_\beta=\zeta_{10}^{-1}$. This allows us to determine $\theta_b|_{\Lh}$ as desired.

For  $\theta_a$, we have $\dim\Lg_i=12$, $i=2,8$, $\dim\Lg_i=13$, $i=0,4,6$, and $\dim\Lg_i=14$, $i=1,3,5,7,9$. This implies that $\on{dim}\Lg_i\cap E=1$ when $i=1,3,4,6,7,9$, and $\on{dim}\Lg_i\cap E=0$ otherwise. Thus $\on{dim}\Lh_i=1$ when $i=0,1,3,4,5,6,7,9$, and $\on{dim}\Lh_2=0$. Apply similar arguments as in the case of $10_b$, we conclude that either $c_\alpha c_{\theta_b\alpha}=\zeta_{10}^3$ and $c_\beta=\zeta_{10}^2$, or $c_\alpha c_{\theta_b\alpha}=\zeta_{10}^{-3}$ and $c_\beta=\zeta_{10}^{-2}$. This allows us to determine $\theta_a|_{\Lh}$ as desired.

More explicitly, we define an order 10 outer automorphism of $H=SL_3$ as follows: 
$$\theta_\gamma:H\to H,\,g\mapsto J^{-1}(g^{t})^{-1}J,\,J=\left(\begin{matrix}&&1\\&1&\\\gamma&&\end{matrix}\right)$$
where $\gamma^5=1$. When $\gamma=\zeta_{10}^2$ (resp. $\zeta_{10}^4$) we obtain $\theta_a|_{H}$ (resp. $\theta_b|_{H}$).

We conclude that $\theta_a|_{H}$ affords no cuspidal character sheaves while $\theta_b|_{H}$ afford a bi-orbital character sheaf supported on the whole $\Lh_1$.
\end{exam}

\begin{exam} The grading $10_b$ of $^2E_6$.  
\normalfont 
The grading $10_b$ of $^2E_6$ can be obtained as restriction of the grading $10_b$ of $E_7$. We continue using the notations in Example~\ref{example-e7}. One can check that $\Lg_{\{\alpha_i\}}=\Lg_{\{-\alpha_i\}}$, $i=1,4,5$, and 
\bern
&&\Lg_{\{\alpha_1\}}\oplus \Lg_{\{\alpha_4\}}\oplus\Lg_{\{\alpha_5\}}\oplus\Lg_{\{\alpha_6\}}\oplus\Lg_{\{-\alpha_6\}}\oplus\Lg_{\{\alpha_2+\alpha_3+\alpha_4\}}\oplus\Lg_{\{-(\alpha_2+\alpha_3+\alpha_4)\}}\oplus\Lg_{\pm\alpha_7}\\&&\oplus \on{span}\{h_{\alpha_1},h_{\alpha_4},h_{\alpha_2+\alpha_3+\alpha_4},h_{\alpha_5},h_{\alpha_6},h_{\alpha_7}\}
\eern form a ($\theta_b$-stable) subalgebra of type $E_6$ with simple roots $\beta_1=\alpha_1$, $\beta_2=\alpha_4$, $\beta_3=\alpha_2+\alpha_3+\alpha_4,\beta_i=\alpha_{i+1},\,i=4,5,6$.

Let us write $\bar\Lg$ for this subalgebra. Then $\bar\fc=\on{span}\{(1-\zeta_{10})h_{\beta_1}+(1-\zeta_{10}^2)h_{\beta_2}+(2-\zeta_{10}-\zeta_{10}^{-1})h_{\beta_3}+(3-\zeta_{10}-\zeta_{10}^{-1})h_{\beta_4}+2h_{\beta_5}+h_{\beta_6}\}$ is a Cartan subspace of $\bar\Lg_1$. We have
\beqn
Z_{\bar\Lg}(\bar\fc)_{\on{der}}=\on{span}\{h_\beta, X_{\pm\beta}\}\eeqn
where $\beta=\alpha_7=\beta_6$. From Example~\ref{example-e7}, we see that $\theta(h_\beta)=h_\beta$, $\theta(X_\beta)=c_\beta X_\beta$, where $c_\beta=\zeta_{10}$ or $\zeta_{10}^{-1}$. Thus $\theta|_{Z_{\bar\Lg}(\bar\fc)_{\on{der}}}$ is as desired.

\end{exam}

\begin{exam}The grading $12_e$ of $E_8$. 

\normalfont We label the simple roots as follows
\beqn
\xymatrix{{\substack{\alpha_1\\\circ}}\ar@{-}[r]&{\substack{\alpha_3\\\circ}}\ar@{-}[r]&{\substack{\alpha_4\\\circ}}\ar@{-}[r]\ar@{-}[d]&{\substack{\alpha_5\\\circ}}\ar@{-}[r]&{\substack{\alpha_6\\\circ}}\ar@{-}[r]&{\substack{\alpha_7\\\circ}}\ar@{-}[r]&{\substack{\alpha_8\\\circ}}\\&&{\substack{\circ\\\alpha_2}}&&}
\eeqn

Let $w=s_1s_4s_6s_2s_3s_5$ (type $E_6$), where $s_i=s_{\alpha_i}$. We have $\dim\Lt_i=1$, $i=1,4,5,7,8,11$, $\dim\Lt_0=2$ and $\dim\Lt_i$=0 otherwise. Moreover,
$$\fc=\Lt_1=\on{span}\{(1,\zeta_{12}^2-\zeta_{12}-1,1+\zeta_{12}^{-1},1+\zeta_{12}^{-1}+\zeta_{12},1+\zeta_{12}^{-1},1,0,0\}$$
and
\bern
&&Z_\Lg(\fc)\cong \Lt\oplus\on{span}\{X_{\pm\beta_i},\,X_{\pm\theta\beta_i},\,X_{\pm\theta^2\beta_i},\,X_{\pm\gamma_i},\,i=1,2,3\}
\eern
where 
\bern
&&\beta_1=12343321,\ \theta\beta_1=12244321,\ \theta^2\beta_1=22343221,\\ &&\beta_2=11221111,\ \theta\beta_2=01122211,\ \theta^2\beta_2=11122111, \\&&\beta_3=11122110,\ \theta\beta_3=11221110,\ \theta^2\beta_3=01122210,\\
 &&\gamma_1=23465431,\ \gamma_2=23465432,\ \gamma_3=00000001\\
 &&\theta^3\beta_i=\beta_i,\,i=1,2,3,\ \theta\gamma_i=\gamma_i\,i=1,2,3.
 \eern
Thus $\fh:=Z_\Lg(\fc)_{\on{der}}$ is of type $D_4$, with simple roots $\beta_3,\theta\beta_3,\theta^2\beta_3$, $\gamma_3$ (center).

It follows that there are $18$ orbits of roots of size $12$ under the action of $w$. 

We have $\dim\Lg_6=18$, $\dim\Lg_i=19$, $i=2,10$, $\dim\Lg_0=22$, $\dim\Lg_i=21$, $i=4,5,7,8$, $\dim\Lg_i=20$, $i=3,9$ and $\dim\Lg_i=23$, $i=1,11$. 

Let $E=\on{span}\{X_{\pm\beta_i},\,X_{\pm\theta\beta_i},\,X_{\pm\theta^2\beta_i},\,X_{\pm\gamma_i}$. Then it follows that
\bern
&&\dim E\cap\Lg_i=2,\,i=0,3,4,5,7,8,9,\,\dim E\cap\Lg_i=1,\,i=2,10,\,\dim E\cap\Lg_i=4,\,i=1,11,\\
&&\text{and $\dim E\cap\Lg_i=0$ otherwise.}
\eern
Note that $h_{\gamma_3}, h_{\beta_3}+h_{\theta\beta_3}+h_{\theta^2\beta_3}\in\Lh_0$, $h_{\beta_3}+\zeta_{12}^4h_{\theta\beta_3}+\zeta_{12}^8h_{\theta^2\beta_3}\in\Lh_4$, and $h_{\beta_3}+\zeta_{12}^8h_{\theta\beta_3}+\zeta_{12}^4h_{\theta^2\beta_3}\in\Lh_8$. 

Using similar argument as before, we conclude that, without loss of generality, $c_{\gamma_1}=c_{\gamma_3}=\zeta_{12}$, $c_{\gamma_2}=\zeta_{12}^2$, $c_{\beta_i}c_{\theta\beta_i}c_{\theta^2\beta_i}=\zeta_{12}^3$, $i=1,2$, and $c_{\beta_3}c_{\theta\beta_3}c_{\theta^2\beta_3}=1$. We can then conclude that $\theta|_{\Lh}$ has the desired Kac diagram. Moreover, $\theta|_{\Lh}$ affords a biorbital cuspidal character sheaf supported on the whole $\Lh_1$.

\end{exam}

\subsection{Grading \texorpdfstring{$2_a$}{2a} of type \texorpdfstring{$E_6$}{E6}}
Let $\theta$ be an automorphism of $E_6$ that induces the grading $2_a$.
\begin{prop}\label{prop-2ae6}
Suppose that $\fc_t\subset \fc$ is a stratum such that $\dim f(\fc_t)=2$. We have the following cases
\begin{enumerate}
\item $Z_\Lg(\fc_t)_{\on{der}}$ is of type $A_2\times A_2$ and $\theta|_{Z_\Lg(\fc_t)_{\on{der}}}$ permutes the two factors of $A_2$.
\item  $Z_\Lg(\fc_t)_{\on{der}}$ is of type $A_1\times A_1\times A_1$, and $\theta|_{Z_\Lg(\fc_t)_{\on{der}}}$ permutes the first two factors of $A_1$ and restricts to the third factor as $\theta_2$.
\item $Z_\Lg(\fc_t)_{\on{der}}$ is of type $A_3$, and $\theta|_{Z_\Lg(\fc_t)_{\on{der}}}$ is the inner automorphism such that $Z_\Lg(\fc_t)_{\on{der}}^\theta$ is of type $A_1\times A_1$.
\item $Z_\Lg(\fc_t)_{\on{der}}$ is of type $A_2$, and $\theta|_{Z_\Lg(\fc_t)_{\on{der}}}$ is the outer automorphism such that $Z_\Lg(\fc_t)_{\on{der}}^\theta\cong\mathfrak{so}_3$.
\end{enumerate}
\end{prop}
\begin{coro}Let $\cF$ be a cuspidal character sheaf on $\Lg_1$ and $\cS$ the supporting stratum of $\cF$. Let $x_s+x_n\in\cS$. Then $Z_\Lg(x_s)_{\on{der}}:=\fh_{x_s}$ is of type $A_2\times A_2$ and $\theta|_{Z_\Lg(x_s)_{\on{der}}}$ permutes the two factors of $A_2$. Moreover, $x_n\in (\fh_{x_s})_1^{\on{nil}}$ lies in the open dense $H_0$-orbit, where $H_0=(Z_G(x_s)_{\on{der}})^\theta$.
\end{coro}
\begin{proof}
Among the cases in Proposition~\ref{prop-2ae6}, only in (1) $\theta|_{Z_\Lg(\fc_t)_{\on{der}}}$ affords a bi-orbital cuspidal character sheaf, on the open dense nilpotent $(Z_G(\fc_t)_{\on{der}})^\theta$-orbit in $(Z_\Lg(\fc_t)_{\on{der}})_1$.
\end{proof}

We label the simple roots as follows
\beqn
\xymatrix{{\substack{\alpha_1\\\circ}}\ar@{-}[r]&{\substack{\alpha_3\\\circ}}\ar@{-}[r]&{\substack{\alpha_4\\\circ}}\ar@{-}[r]\ar@{-}[d]&{\substack{\alpha_5\\\circ}}\ar@{-}[r]&{\substack{\alpha_6\\\circ}}\\&&{\substack{\circ\\\alpha_2}}&&}.
\eeqn
Let \beqn
w_h=s_{\alpha_1}s_{\alpha_4}s_{\alpha_6}s_{\alpha_3}s_{\alpha_5}s_{\alpha_2} \text{ (a Coxeter element in $W(E_6)$)}.
\eeqn
As before we fix $\ft\supset\fc$. The involution $2_a$ can be realised as $\theta=\on{Int}n_w$, where $w=w_h^6$. We have
\bern
w\alpha_1=-\alpha_6,w\alpha_2=-\alpha_2,w\alpha_3=-\alpha_5,w\alpha_4=-\alpha_4, w\alpha_5=-\alpha_3,w\alpha_6=-\alpha_1.
\eern
Moreover,
\beqn
\fc=\on{span}\{\bar h_1:=h_{\alpha_1}+h_{\alpha_6},\bar h_2:=h_{\alpha_3}+h_{\alpha_5},\bar h_3:=h_{\alpha_4},\bar h_4:=h_{\alpha_2}\}
\eeqn
and  the little Weyl group
\beqn
W\cong\langle t_1:=s_{\alpha_1}s_{\alpha_6},t_2=s_{\alpha_3}s_{\alpha_5},t_3:=s_{\alpha_4},t_4:=s_{\alpha_2}\rangle\cong W(F_4).
\eeqn
Recall that  the restricted roots, i.e, $\bar\alpha:=\alpha|_{\fc}$, form a root system of type $F_4$. Moreover, the real roots $\alpha$, i.e, those with $\theta\alpha=-\alpha$, restricts to the long roots of  $F_4$, and the complex roots $\alpha$, i.e, those with $\theta\alpha\neq\pm\alpha$, restricts to the short roots of $F_4$ with $\bar\alpha=-\overline{\theta\alpha}$.

Let us write $\fc_{t_i,t_j}\subset\fc$ for the subspace fixed by $t_i,t_j$.

Suppose that $\fc_t\subset \fc$ is a stratum such that $\dim f(\fc_t)=2$. Let us write $\bar \Lg=F_4$ and $\Lh=Z_\Lg(\fc_t)_{\on{der}}$. Viewing $\fc$ as a Cartan subalgebra of $\bar\Lg$, we have the following possibilities:
\begin{enumerate}
\item $Z_{\bar \Lg}(\fc_t)$ is a Levi subgroup of type $A_2$ (long roots). Then $Z_\Lg(\fc_t)$ is conjugate to $$Z_{\Lg}(\fc_{t_3,t_4})=\Lt\oplus\on{span}\{X_{\pm\alpha_2+\alpha_4},X_{\pm\alpha_2},X_{\pm\alpha_4}\}\text{ (type $A_2$)}$$ with $\dim\Lh_0=3,\dim\Lh_1=5$. In this case $\theta|_\Lh$ is the involution with $\Lh^\theta\cong\mathfrak{so}(3)$.
\item $Z_{\bar \Lg}(\fc_t)$ is a Levi subgroup of type $\tilde A_2$ (short roots). Then $Z_\Lg(\fc_t)$ is conjugate to $$Z_{\Lg}(\fc_{t_1,t_2})=\Lt\oplus\on{span}\{X_{\pm\alpha_5+\alpha_6},X_{\pm\alpha_1+\alpha_3},X_{\pm\alpha_1},X_{\pm\alpha_3},X_{\pm\alpha_5},X_{\pm\alpha_6}\}\text{
(type $A_2\times A_2$)},$$  $\theta$ permutes the two factors, $\dim\Lh_0=\dim\Lh_1=8$.
\item $Z_{\bar \Lg}(\fc_t)$ is a Levi subgroup of type $A_1\times\tilde A_1$. Then $Z_\Lg(\fc_t)$ is conjugate to $$Z_{\Lg}(\fc_{t_2,t_4})=\Lt\oplus\on{span}\{X_{\pm\alpha_2},X_{\pm\alpha_3},X_{\pm\alpha_5}\} \text{ (type $A_1\times A_1\times A_1$)},$$ $\theta$ permutes the last two factors and leaves the first one stable, $\dim\Lh_0=4,\dim\Lh_1=5$. In this case $\theta$ restricts to the first factor of $A_1$ as $\theta_2$.
\item $Z_{\bar \Lg}(\fc_t)$ is a Levi subgroup of type $B_2$. Then $Z_\Lg(\fc_t)$ is conjugate to $$Z_{\Lg}(\fc_{t_2,t_3})=\Lt\oplus\on{span}\{X_{\pm\alpha_4+\alpha_5},X_{\pm\alpha_3+\alpha_4+\alpha_5},X_{\pm\alpha_3+\alpha_4},X_{\pm\alpha_3},X_{\pm\alpha_4},X_{\pm\alpha_5}\}\text{ (type $A_3$)}$$ 
 with $\dim\Lh_0=7,\dim\Lh_1=8$. In this case $\theta|_{\Lh}$ gives rise to the quasi-split symmetric pair for $SL_4$.

\end{enumerate}
Thus Proposition~\ref{prop-2ae6} follows.

\begin{coro}\label{cor:E6}
We have $N_{G_0}(\fc_{t_1,t_2})/Z_{G_0}(\fc_{t_1,t_2})\cong W(G_2)$, the Weyl group of type $G_2$. Moreover,~\autoref{theo:cuspidal-cartan} holds in this case.
\end{coro}
\begin{proof}
The Weyl group of $(G_0,\Lg_1)$ is $W(F_4)$, the Weyl group of type $F_4$. Note that $Z_K(\fc_{t_1,t_2})$ is of type $A_2$. The corollary follows from~\cite[Table 2]{ABDR18}.
\end{proof}

\begin{rema}
We note that in case (2), we have $$Z(H)^\theta=Z(G)=\langle\check\alpha_1(\zeta_3)\check\alpha_3(\zeta_3^2)\check\alpha_5(\zeta_3)\check\alpha_6(\zeta_3^2)\rangle\cong\mu_3$$ and it acts non-trivially on the biorbital cuspidal character sheaves.
\end{rema}

Let $\Lh=Z_{\Lg}(\fc_{t_1,t_2})_{\on{der}}$. Let $x_0=X_{\alpha_1}-c_{\alpha_1}X_{-\alpha_6}+X_{\alpha_3}-c_{\alpha_3}X_{-\alpha_5}\in\Lh_1^{\on{nil}}$. Then $H_0.x_0$ is open dense in $\Lh_1^{\on{nil}}$. We see that
\beqn
\fc_{t_1,t_2}+x_0\subset Z_{\Lg_1}(x_0).
\eeqn
Let $\cS$ be the supporting stratum of a cuspidal character sheaf $\cF$. Then we conclude that
\beqn
\cS=\widecheck\cO_{x_0}.
\eeqn

Let $y_0=2(X_{-\alpha_1}-c_{-\alpha_1}X_{\alpha_6}+X_{-\alpha_3}-c_{-\alpha_5}X_{\alpha_5})\in\Lh_1$.
We have $[x_0,y_0]=2(h_{{\alpha}_1}-h_{\alpha_6}+h_{\alpha_3}-h_{\alpha_5}):=h_0$, and $\phi=(x_0,y_0,h_0)$ is a normal $\mathfrak{sl}_2$-triple. 

Let $\Lg^\phi=Z_\Lg(h_0)\cap Z_\Lg(x_0)$. We can check that $\Lg^\phi$ is of type $G_2$ by showing that $h_0$ is $G$-conjugate to $4h_1+4h_2+6h_3+8h_4+6h_5+4h_6$,  where $h_i=h_{\alpha_i}$, so that the weighted Dynkin diagram is $\begin{array}{ccccc}2&0&0&0&2\\&&0&&\end{array}$ (see, for example, tables in~\cite[\S13.1]{C93}). Since $\fc_{t_1,t_2}\subset\Lg^\phi$ has dimension 2, we conclude that $\theta|_{\Lg^\phi}$ is the GIT stable $\bZ/2\bZ$ grading of $G_2$ (cf~\cite[Table IX]{D}). 
The fact that $\cO_{x_0}$ is the unique nilpotent orbit such that $\Lg^\phi\cong G_2$ follows from loc.cit.

\subsection{Grading \texorpdfstring{$3_a$}{3a} of type \texorpdfstring{$E_7$}{E7}}Let $\theta$ be an automorphism of $E_7$ that induces the grading $3_a$.

\begin{prop}
Suppose that $\fc_t\subset \fc$ is a stratum such that $\dim f(\fc_t)=2$. We have the following cases
\begin{enumerate}

\item  $Z_\Lg(\fc_t)_{\on{der}}$ is of type $A_1\times A_1\times A_1$, and $\theta|_{Z_\Lg(\fc_t)_{\on{der}}}$ permutes the three factors of $A_1$.
\item $Z_\Lg(\fc_t)_{\on{der}}$ is of type $A_2$, and $\theta|_{Z_\Lg(\fc_t)_{\on{der}}}$ is the GIT stable grading of order $3$.
\end{enumerate}
\end{prop}
\begin{coro}Let $\cF$ be a cuspidal character sheaf on $\Lg_1$ and $\cS$ the supporting stratum of $\cF$. Let $x_s+x_n\in\cS$. Then $Z_\Lg(x_s)_{\on{der}}:=\fh_{x_s}$ is of type $A_1\times A_1\times A_1$ and $\theta|_{Z_\Lg(x_s)_{\on{der}}}$ permutes the three factors of $A_1$. Moreover, $x_n\in (\fh_{x_s})_1^{\on{nil}}$ lies in the open dense $H_0$-orbit, where $H_0=Z_G(x_s)_{\on{der}}^\theta$.
\end{coro}

We have $\dim\Lg_0=43$, and $\dim\Lg_i=45$, $i=1,2$. Let $\theta=\on{Int}n_w$, where 
$$w=s_{\alpha_0}s_{\alpha_2}s_{\alpha_6}s_{\alpha_1}s_{\alpha_4}s_{\alpha_7} \text{ (of type $3A_2$)}.$$
We have
\bern
w:&&\alpha_1\mapsto1234321,\alpha_2\mapsto\alpha_4,\ \alpha_3\mapsto -1123321,\ \alpha_4\mapsto-0101000,\\ 
&&\alpha_5\mapsto0101110,\ \alpha_6\mapsto\alpha_7,\ \alpha_7\mapsto-0000011.
\eern
Moreover,
\bern
&&\Lt_1=\on{span}\{v_1=({
  {1, 0, 1-\zeta_3, 2, 1-\zeta_3, 0, 1+\zeta_3}
 }), v_2=({
  {0, 1, 0, -\zeta_3, 0, 0, 0}
 }), \\&&\hspace{2in}v_3=({
  {0, 0, 0, 0, 0, 1, -\zeta_3}
 })\}.
\eern
The roots form $42$ orbits of size $3$ under the action of $w$. 
We can find a set of generators for $W\cong G_{26}$ as follows
\beqn
w_1=s_{1111000}s_{0011000}s_{1224321}, \,w_2=s_{1223210}s_{1111100},\,w_3=s_{1122210}s_{1011100}
\eeqn
with $w_1^2=1$, $w_2^3=w_3^3=1$.

Suppose that $\fc_t\subset\fc$ is a stratum such that $\dim f(\fc_t)=2$. Then $Z_\Lg(\fc_t)$ is conjugate to either $Z_\Lg(\fc_{w_1})$ or $Z_\Lg(\fc_{w_2})$, where $\fc_{w_i}$ is the subspace fixed by $w_i$.

We check that
\bern
&&\fc_{w_1}=\on{span}\{v_1 -\zeta_3 v_2, v_3\},\ \fc_{w_2}=\on{span}\{ v_1-\zeta_3 v_3, v_2 + (1+\zeta_3) v_3\}
\eern
and
\bern
&&Z_\Lg(\fc_{w_1})=\Lt\oplus\on{span}\{X_{\pm1111000},X_{\pm0011000},X_{\pm1224321}\}\text{ (type $A_1\times A_1\times A_1$)}\\
&&Z_\Lg(\fc_{w_2})=\Lt\oplus\on{span}\{X_{\pm1223210},X_{\pm0112110},X_{\pm1111100}\}\text{ (type $A_2$)}
\eern
where
\bern
&&\theta({1111000})={0011000},\ \theta({0011000})={-1224321},\\
&&\theta(1111100)=0112110,\ \theta(0112100)=-1223210.
\eern
Thus we have
\begin{enumerate}
\item in the first case  $\theta$ permutes the three factors of $A_1$ with $\dim\Lh_i=3$, $i=0,1,2$, where $\Lh:=Z_\Lg(\fc_{w_1})_{\on{der}}$. 
\item in the second case $\theta|_{\Lh}$ as the order 3 GIT stable grading with $\dim\Lh_0=2$, $\dim\Lh_1=\dim\Lh_2=3$, where $\Lh:=Z_\Lg(\fc_{w_2})_{\on{der}}$.
\end{enumerate}

\begin{coro}\label{coro:E7}
We have $N_{G_0}(\fc_{w_1})/Z_{G_0}(\fc_{w_1})\cong G_5$. Moreover,~\autoref{theo:cuspidal-cartan} holds in this case.
\end{coro}
\begin{proof}
The Weyl group of $(G_0,\Lg_1)$ is the complex reflection group $G_{26}$. Note that $Z_K(\fc_{w_1})$ is of type $A_1$.  The corollary follows from~\cite[Table 1]{ABDR18}.
\end{proof}

\begin{rema}
We note that in case (1), we have $$Z(H)^\theta=Z(G)=\langle\check\alpha_2(-1)\check\alpha_5(-1)\check\alpha_7(-1)\rangle\cong\mu_2$$ and it acts non-trivially on the biorbital cuspidal character sheaf.
\end{rema}
Let $\Lh=Z_\Lg(\fc_{w_1})_{\on{der}}$. Let $x_0=X_{\alpha}+\zeta_3c_\alpha X_{\theta\alpha}+\zeta_3^2 c_\alpha c_{\theta\alpha}X_{\theta^2\alpha}\in\Lh_{-1}^{\on{nil}}$, where $\alpha=1111000$. We see that
\beqn
\fc_{w_1}+X_{\alpha}+\zeta_3^2c_\alpha X_{\theta\alpha}+\zeta_3 c_\alpha c_{\theta\alpha}X_{\theta^2\alpha}\subset Z_{\Lg_1}(x_0).
\eeqn
We conclude that $\cS=\widecheck{\cO}_{x_0}$, where $\cS=\on{supp}\cF$ for cuspidal character sheaves $\cF$. Let $y_0=X_{-\alpha}+\zeta_3^2c_{-\alpha} X_{-\theta\alpha}+\zeta_3 c_{-\alpha} c_{-\theta\alpha}X_{-\theta^2\alpha}\in\Lh_{1}$. Then $[x_0,y_0]=h_\alpha+h_{\theta\alpha}+h_{\theta^2\alpha}=-h_2-2h_4-3h_5-2h_6-h_7:=h_0\in\Lh_0$, where $h_i=h_{\alpha_i}$, and $\phi=\{x_0,y_0,h_0\}$ is a normal $\mathfrak{sl}_2$-triple. 

Let $\Lg^\phi=Z_\Lg(x_0)\cap Z_\Lg(h_0)$. Then we can check that $\Lg^\phi$ is of type $F_4$ by showing that $h_0$ is $G$-conjugate to $2h_1+3h_2+4h_3+6h_4+5h_5+7h_6+3h_7$ , so that the weighted Dynkin diagram is $\begin{array}{cccccc}0&0&0&0&0&2\\&&0&&\end{array}$(see, for example, tables in~\cite[\S13.1]{C93}).
Since $\fc_{w_1}\subset\Lg^\phi$ has dimension 2, we conclude that $\theta|_{\Lg^\phi}$ is a GIT stable $\bZ/3\bZ$-grading.

The nilpotent orbits in $\Lg_1$ (equivalently, $\Lg_{-1}$) has been classified in~\cite{GT}. We can identify the orbit $\cO_{x_0}$ with the orbit \#71 in~\cite[Table 8]{GT} since $\Lg_0^\phi$ is of type $2A_2$. The uniqueness of such a nilpotent orbit then follows from loc. cit.

\printbibliography

\end{document}